\documentclass[11pt]{amsart}

\usepackage{amsfonts,amsthm,amsmath,enumerate,amssymb,latexsym,color,tcolorbox,tikz,mathrsfs,bm}
\usetikzlibrary{patterns}
\usepackage{subfig}
\usepackage[text={6in,9in},centering]{geometry}
\usepackage[colorlinks,
            linkcolor=blue,
            anchorcolor=red,
            citecolor=red]{hyperref}

\graphicspath{{figures/}}

\theoremstyle{plain}
\newtheorem{theorem}{Theorem}[section]
\newtheorem{lemma}[theorem]{Lemma}
\newtheorem{corollary}[theorem]{Corollary}
\newtheorem{proposition}[theorem]{Proposition}

\theoremstyle{definition}
\newtheorem{definition}[theorem]{Definition}
\newtheorem{example}[theorem]{Example}

\theoremstyle{remark}
\newtheorem{remark}[theorem]{Remark}

\numberwithin{equation}{section}

\DeclareMathOperator{\dist}{dist}

\def\R{\mathbb R}
\def\Z{\mathbb Z}
\def\D{\mathcal D}

\def\P{\mathcal P}
\def\C{\mathscr C}
\def\i{\bm i}
\def\j{\bm j}

\def\bi{\mathbf i}

\def\id{\textup{id}}

\def\vp{\varphi}

\begin{document}

\title[Existence of cut points of connected GSCs]{On the existence of cut points of connected \\ generalized Sierpi\'nski carpets}

\author{Huo-Jun Ruan}
\address{School of Mathematical Sciences, Zhejiang University, Hangzhou, China}
\email{ruanhj@zju.edu.cn}

\author{Yang Wang}
\address{Department of Mathematics, The Hong Kong University of Science and Technology, Clear Water Bay, Kowloon, Hong Kong}
\email{yangwang@ust.hk}

\author{Jian-Ci Xiao}
\address{School of Mathematical Sciences, Zhejiang University, Hangzhou, China}
\email{jcxshaw24@zju.edu.cn}

\subjclass[2010]{Primary 28A80; Secondary 54A05}

\keywords{Generalized Sierpi\'nski carpets, cut points, connectedness, Hata graphs.}

\thanks{Corresponding author: Jian-Ci Xiao}

\begin{abstract}
    In a previous work joint with Dai and Luo, we show that a connected generalized Sierpi\'nski carpet (or shortly a GSC) has cut points if and only if the associated $n$-th Hata graph  has a long tail for all $n\geq 2$. In this paper, we extend the above result by showing that it suffices to check a finite number of those graphs to reach a conclusion. This criterion provides a truly ``algorithmic'' solution to the cut point problem of connected GSCs. We also construct for each $m\geq 1$ a connected GSC with exactly $m$ cut points and demonstrate that when $m\geq 2$, such a GSC must be of the so-called fragile type.
\end{abstract}

\maketitle

\section{Introduction}

A large amount of common fractal sets are totally disconnected or at least have infinitely many connected components. But there are indeed some of them which are connected (e.g., the standard Sierpi\'nski carpet). For a given pair of connected fractals, an interesting question is whether they are mutually homeomorphic. In~\cite{Why58}, Whyburn came up with an elegant characterization concerning some special cases: A metrizable topological space is homeomorphic to the standard Sierpi\'nski carpet if and only if it is a locally connected planar continuum of topological dimension 1 that has no local cut points. Recall that for any connected topological space $X$, $x\in X$ is called a (global) \emph{cut point} of $X$ if $X\setminus\{x\}$ is disconnected, and is called a \emph{local cut point} if $x$ is a cut point of some connected neighborhood of itself.

When the given connected fractal is a self-affine set, a well-known result of Hata~\cite{Hat85} guarantees that it is also locally connected. Combining with Whyburn's result, the existence of local cut points becomes the key to determine whether a given planar connected self-affine set is homeomorphic to the standard Sierpi\'nski carpet. However, it appears that developing a systematic approach for detecting the existence of local cut points, or even cut points, will be a difficult task. In~\cite{ALT21}, Akiyama, Loridant and Thuswaldner characterize the existence of cut points of a special class of self-affine tiles. They considered the self-affine tile $T=\bigcup_{d\in\mathcal{D}} M^{-1}(T+d)$, where $M$ is a integral matrix with characteristic polynomial $x^2+ax+b$, where $a,b\in\mathbb{Z}$ are such that $|a|<b$ and $b\geq 2$, and $\mathcal{D}=\{0,v,2v,\ldots, (b-1)v\}$ for some vector $v\in\mathbb{R}^2$ such that $v, Mv$ are linearly independent. It is showed in~\cite{ALT21} that $T$ has cut points if and only if $2|a|\leq b+2$. In~\cite{Wen22}, Wen studies the existence of cut points of the so-called fractal necklaces. In particular, it is showed that some special classes of fractal necklaces contain no cut points. Interested readers can refer to~\cite{DLRWX21} for an introduction to earlier relevant studies.

Together with Dai and Luo, the authors provided a criterion in~\cite{DLRWX21} on the existence of (local) cut points of a special class of planar self-similar sets called the generalized Sierpi\'nski carpets, which are defined as follows. Let $N\geq 2$ and let $\D\subset\{0,1,\ldots,N-1\}^2$ be a non-empty digit set with $1<|\D|<N^2$ (to avoid trivial cases), where $|\D|$ denotes the number of elements in $\D$. For each $i\in\D$, define a similarity map $\vp_i$ by
\[
    \vp_i(x) = \frac{1}{N}(x+i), \quad x\in\R^2.
\]
We call the self-similar set $F=F(N,\D)$ associated with the iterated function system (IFS for short) $\{\vp_i:i\in\D\}$ a \emph{generalized Sierpi\'nski carpet} (or abbreviated to GSC).

For convenience, we regard the digit set $\D$ as the index set of the IFS $\{\vp_i:i\in\D\}$ instead of enumerating it by $\{\vp_1,\ldots,\vp_{|\D|}\}$. Under this setting, the following notations are typically used.

\begin{enumerate}
    \item For $k\in\Z^+$, $\D^k:=\{\i=i_1\cdots i_k: i_1,\ldots,i_k\in\D\}$. Let $\D^0=\{\vartheta\}$, where $\vartheta$ denotes the \emph{empty word}. For $k\geq 0$ and $\i\in \D^k$, we call $\i$ a word of length $|\i|:=k$.
    \item Let $\D^*=\bigcup_{k=1}^\infty\D^k$ and $\D^\infty=\{i_1i_2\cdots: i_j\in\D \text{ for all } j\in\Z^+\}$ denote the collection of finite and infinite words, respectively;
    \item For $\i\in\D^*$, we call $\vp_{\i}(F)$ a level-$|\i|$ \emph{cell};
    \item For $1\leq k\leq n$ and $\i=i_1\cdots i_n\in\D^n$, write $\i|_k:=i_1\cdots i_k$ to be the prefix of $\i$ of length $k$. For $\bi\in\D^\infty$ and $k\geq 1$, $\bi|_k$ is similarly defined;
    \item For $\i,\j\in\D^*$, write $\i\prec\j$ whenever $\i$ is a prefix of $\j$, and $\i\nprec\j$ otherwise;
    \item For $\i\in\D^*$ and $k\in\Z^+$, $\i\D^k:=\{\i\j: \j\in\D^k\}$.
    \item For $k\geq 1$ and $\i=i_1\cdots i_k\in\D^k$, $\vp_{\i}:=\vp_{i_1}\circ\cdots \circ\vp_{i_k}$.
    Also denote by $\vp_{\j}^k$ the $k$-fold composition of $\vp_{\j}$ for $\j\in \D^*$.
    \item For $\i\in\D^*$ and $n\geq 1$, $\i^n:=\underbrace{\i\cdots\i}_{n \text{ terms}}$.
\end{enumerate}

As in~\cite{DLRWX21}, for any graph $G$ and any vertex $v$, we denote by $G-\{v\}$ the subgraph of $G$ obtained by deleting $v$ and all edges incident with $v$. We call $v$ a \emph{cut vertex} of $G$ if the subgraph $G-\{v\}$ is disconnected. A \emph{connected component} of $G$ is just a maximal connected subgraph of $G$.

The criterion of the existence of cut points in~\cite{DLRWX21} is based on an examination on the associated Hata graph sequence of $F$. For $n\geq 1$, the $n$-th \emph{Hata graph} $\Gamma_n=\Gamma_n(N,\D)$ of $F$ is defined by setting the vertex set to be $\D^n$, and demanding that there is an edge joining $\i,\j\in\D^n$ ($\i\neq\j$) if and only if $\vp_{\i}(F)\cap \vp_{\j}(F)\neq\varnothing$.

\begin{definition}[\cite{DLRWX21}]\label{def:chi}
    Let $G=(V,E)$ be a connected graph.  Given a cut vertex $v\in V$ of $G$, let $G_1(v),\ldots,G_m(v)$ be all connected components of $G-\{v\}$ with $|G_1(v)|\geq |G_2(v)|\geq \cdots \geq |G_m(v)|$, where $|G_i(v)|$ stands for the number of vertices in $G_i(v)$, $1\leq i\leq m$. Define
    \[
        \chi(G) = \max\{|G_2(v)|:\, \textrm{$v$ is a cut vertex of $G$}\}
    \]
    if $G$ has cut vertices, and $\chi(G)=0$ if $G$ has none.
\end{definition}

Sometimes it is convenient to say that the $n$-th Hata graph $\Gamma_n$ has a \emph{long tail} if $\chi(\Gamma_n)\geq |\D|^{n-1}-1$. We have shown in~\cite{DLRWX21} that a GSC has cut points if and only if the correpsonding $\Gamma_n$ has a long tail for all $n$.

\begin{theorem}[\cite{DLRWX21}]\label{thm:premain}
    A connected GSC $F=F(N,\D)$ has cut points if and only if $\chi(\Gamma_n)\geq|\D|^{n-1}-1$ for all $n\geq 2$.
\end{theorem}

There is one particular type of connected GSCs of which the existence of cut points is relatively easy to determine. More precisely, a connected GSC $F=F(N,\D)$ is called \emph{fragile} if there is a decomposition of $\D$, say $\D=\D_1\cup\D_2$ with $\D_1\cap\D_2=\varnothing$, such that the intersection
\[
    \Big( \bigcup_{i\in\D_1}\vp_i(F) \Big) \cap \Big( \bigcup_{i\in\D_2}\vp_i(F) \Big)
\]
is a singleton. That is to say, one can divide all level-$1$ cells into two groups meeting at one single point. Clearly, that singleton is a cut point, and we do realize an easily checked criterion for this type of GSCs (please see~\cite[Theorem 3.6]{DLRWX21}). A connected GSC is called \emph{non-fragile} if it is not fragile. Unfortunately, for non-fragile cases, Theorem~\ref{thm:premain} requires us to examine the whole sequence $\{\Gamma_n\}_{n=1}^\infty$. It is of particular interest to ask whether one can detect the existence of cut points by checking only a small section of that sequence. This is the main topic of the paper, and we will show that the answer is affirmative.

Our main result is:

\begin{theorem}\label{thm:goodcp}
    Let $F=F(N,\D)$ be a non-fragile connected GSC. Then there is some $M\geq 2$ independent of $N$ such that $F$ has cut points if and only if $\chi(\Gamma_M)\geq|\D|^{M-1}$.
\end{theorem}

More precisely, one can see from later proof that the constant $M=3^8+3$ will suffice. In addition, we look into the possible number of cut points. For simplicity, denote by $C_F$ the set of cut points of any given connected GSC $F$.

\begin{theorem}\label{thm:number}
    For every $m\in\Z^+$, there is some connected GSC $F$ with $\# C_F=m$.
\end{theorem}
\begin{proof}
    Please see Example~\ref{exa:exactlyncutpts} for the construction.
\end{proof}

We remark that all GSCs constructed in Example~\ref{exa:exactlyncutpts} are fragile. This leads to a natural question: \emph{Given any $m\in\Z^+$, is there a non-fragile connected GSC $F$ with $\#C_F=m$?} When $m=1$, the GSC in Figure~\ref{fig:numexa}(A) is as required. However, the situation is completely different when $m\geq 2$.
More precisely, we have the following observation.

\begin{theorem}\label{thm:2orinfinity}
    If $F$ is non-fragile and $\#C_F\geq 2$, then $\#C_F=+\infty$.
\end{theorem}

Clearly, if two sets are homeomorphic then they have the same number of cut points. Therefore, the above theorem indicates that for a fragile GSC $F$ with $2\leq \#C_F<\infty$, $F$ is not homeomorphic to any non-fragile GSC.

The organization of this paper is as follows. In Section 2, we collect some preliminary results in~\cite{DLRWX21} which will be used later. In Section 3, we record several basic observations. Sections 4 and 5 are devoted to properties of essential cut vertices of Hata graphs and the proof of Theorem~\ref{thm:goodcp}, respectively. Section 6 presents the detailed proof (which is a bit complicated) of an auxiliary proposition in the proof of Theorem~\ref{thm:goodcp}. Finally, we construct some interesting examples, which demonstrate Theorem~\ref{thm:number}, in Section 7.1 and prove Theorem~\ref{thm:2orinfinity} in Section 7.2.

\section{Preliminary properties}


For reader's convenience, we invoke here some results in~\cite{DLRWX21} which will be used later.

\begin{lemma}[{\cite[Proposition 3.2]{DLRWX21}}]\label{lem:evenfragile}
    Suppose that there is some $m\geq 1$ such that $\D^m$ can be decomposed as $\D^m=I\cup J$ with $I\cap J=\varnothing$ and
    \[
        \Big( \bigcup_{\i\in I}\vp_{\i}(F) \Big) \cap \Big( \bigcup_{\i\in J}\vp_{\i}(F) \Big) = \{x\}
    \]
    for some $x\in F$. Then $F$ is fragile.
\end{lemma}

\begin{lemma}[{\cite[Proposition 3.7]{DLRWX21}}]\label{lem:casedisfact}
    Let $m,k\geq 1$ and let $\i,\j\in\D^m$ be two distinct words. If there exists exactly one pair of $\i',\j'\in\D^k$ such that
    $\vp_{\i\i'}(F) \cap \vp_{\j\j'}(F) \neq\varnothing$, then $\vp_{\i}(F)\cap\vp_{\j}(F)$ is a singleton.
\end{lemma}

\begin{lemma}[{\cite[Lemma 3.9]{DLRWX21}}]\label{lem:fourvertex}
    Let $\alpha\in\{(0,0),(N-1,0),(0,N-1),(N-1,N-1)\}$ and let $\i,\j\in\D^*$ with $|\i|=|\j|$. If $\vp_{\j}(F)\cap\vp_{\i\alpha}(F)\neq\varnothing$, then $\vp_{\i}(\frac{\alpha}{N-1})$, which is a vertex of the square $\vp_{\i}([0,1]^2)$, is an element of $\vp_{\i}(F)\cap\vp_{\j}(F)$.
\end{lemma}


\begin{lemma}[{\cite[Lemma 4.6]{DLRWX21}}]\label{lem:equivalent}
    Let $n\geq 1$ and let $\j=j_1\cdots j_n\in\D^n$. Let $\omega,\tau\in\D\setminus\{j_1\}$. Then $\omega\D^{n-1}$ and $\tau\D^{n-1}$ belong to different connected components of $\Gamma_{n}-\{\j\}$ if and only if $\vp_{\omega}(F)$ and $\vp_\tau(F)$ belong to different connected components of $\bigcup_{\eta\in\D^n\setminus\{\j\}} \vp_{\eta}(F)$.
\end{lemma}

\begin{lemma}[{\cite[Corollary 4.7]{DLRWX21}}]\label{lem:nonfragilelast}
    Let $k\geq 1$ and let $\j\in\D^k$ be a cut vertex of $\Gamma_k$. Suppose there are $\omega,\tau\in\D$ such that $\omega\D^{k-1}$ and $\tau\D^{k-1}$ belong to different connected components of $\Gamma_k-\{\j\}$. Then $\omega\D^{q}$ and $\tau\D^q$ belong to different connected components of $\Gamma_{q+1}-\{\j|_{q+1}\}$ for all $0\leq q<k$.
\end{lemma}

\begin{lemma}[{\cite[Lemma 4.9]{DLRWX21}}]\label{lem:connected}
    Let $m\geq 2$ and let $B_1,\ldots,B_m$ be connected compact sets in $\R^2$ such that $\bigcup_{i=1}^m B_i$ is also connected. If $A\subset B_1$ satisfies that $B_1\setminus A$ remains connected and $A\cap B_i=\varnothing$ for all $i\neq 1$, then $(B_1\setminus A)\cup B_2\cup\cdots\cup B_m$ is also connected.
\end{lemma}

\section{Basic observations on graphs and GSCs}

In the rest of this paper, $F=F(N,\D)$ is presumed to be any fixed connected non-fragile GSC. We will record in this section some useful observations.

\begin{lemma}\label{lem:graphcon}
    Let $G$ be a connected graph. If $v_0$ is a cut vertex of $G$, then every connected component of $G-\{v_0\}$ contains at least one vertex adjacent to $v_0$.
\end{lemma}
\begin{proof}
    Suppose that there is some connected component $C$ of $G-\{v_0\}$ that contains no neighbors of $v_0$. Denote the vertex set of $C$ by $V_C$. Then for any $v\in V_C$ and $v'\notin V_C$, $v, v'$ are not adjacent. But this implies that $G$ has at least two connected components, which contradicts its connectedness.
\end{proof}

\begin{lemma}\label{lem:notacutpt}
    Let $(a,0)\in\D$ for some $0\leq a\leq N-1$. If $(a-1,0)\notin\D$ and $(a+1,0)\notin\D$, then $(a,0)$ is not a cut vertex of $\Gamma_1$.
\end{lemma}
\begin{proof}
    Suppose on the contrary that $\Gamma_1-\{(a,0)\}$ is disconnected. By Lemma~\ref{lem:graphcon}, there are two neighbor vertices of $(a,0)$ which belong to different connected components of $\Gamma_1-\{(a,0)\}$. Note that $(a-1,1),(a,1),(a+1,1)$ are the only three possible vertices that might be adjacent to $(a,0)$.

    We claim that at least one of them is not a neighbor of $(a,0)$. Otherwise, since $\vp_{(a-1,1)}(F)\cap\vp_{(a,0)}(F)\neq\varnothing$ and $\vp_{(a+1,1)}(F)\cap\vp_{(a,0)}(F)\neq\varnothing$, it is easy to see that
    \[
        \{(0,0),(N-1,0),(0,N-1),(N-1,N-1)\} \subset \D.
    \]
    Equivalently, $\{(0,0),(1,0),(0,1),(1,1)\}\subset F$. So
    \[
        \vp_{(a-1,1)}(F) \cap \vp_{(a,1)}(F) \supset \{\vp_{(a-1,1)}((1,0))\} = \{\vp_{(a,1)}((0,0))\}
    \]
    and hence the vertices $(a-1,1),(a,1)$ are adjacent in $\Gamma_1$. Similarly, $(a,1)$ and $(a+1,1)$ are also adjacent. As a result, $(a,0)$ is not a cut vertex of $\Gamma_1$. This is a contradiction.

    The claim implies that $(a,0)$ has exactly two neighbors in $\Gamma_1$. Without loss of generality, assume them to be $(a-1,1)$ and $(a,1)$ (other two cases can be similarly discussed). Note that $\Gamma_1-\{(a,0)\}$ has exactly two connected components such that one contains $(a-1,1)$ and the other contains $(a,1)$. Denoting the vertex set of the former component by $\D_1$, we see that
    \[
        \Big( \bigcup_{i\in\D_1}\vp_i(F) \Big) \cap \Big( \bigcup_{i\in\D\setminus\D_1}\vp_i(F) \Big) = \Big( \bigcup_{i\in\D_1}\vp_i(F) \Big) \cap \vp_{(a,0)}(F) = \vp_{(a-1,1)}(F) \cap \vp_{(a,0)}(F),
    \]
    which is a singleton. As a result, $F$ is fragile and we again obtain a contradiction.
\end{proof}

\begin{lemma}\label{lem:iiawayfromj}
    Let $n\geq 1$ and let $\i\in\D^n$. Then $\vp_{\j}(F)\cap\vp_{\i\i}(F)=\varnothing$ for any $\j\in\D^n\setminus\{\i\}$. Moreover, we have for all $k\geq 2$ that
    \[
        \vp_{\j}(F) \cap \vp_{\i^k}(F)=\varnothing, \quad \forall \j\in\D^{(k-1)n}\setminus\{\i^{k-1}\}.
    \]
\end{lemma}

In fact, this is a simple geometric observation. However, to avoid drawing too many figures, we present here a proof based on standard computation.


\begin{proof}
    For the first statement, write $\i=i_1\cdots i_n$. Since $F\subset[0,1]^2$, it suffices to show that
    \[
        \vp_{j_1\cdots j_{2n}}([0,1]^2) \cap \vp_{\i\i}([0,1]^2) = \varnothing
    \]
    for any $j_1\cdots j_{2n}\in\D^{2n}$ with $j_1\cdots j_n\neq\i$. Since the above two squares are both of side length $N^{-2n}$, it suffices to show that $|\vp_{j_1\cdots j_{2n}}((0,0))-\vp_{\i\i}((0,0))|>\sqrt{2}N^{-2n}$.

    Writing $i_k=(i_k^1, i_k^2)$ and $j_k=(j_k^1,j_k^2)$ for all $k$,
    \begin{align*}
        &|\vp_{j_1\cdots j_{2n}}((0,0)) - \vp_{\i\i}((0,0))| \\
        =& \sqrt{\Big( \sum_{k=1}^n\frac{j_k^1-i_k^1}{N^{k}} + \sum_{k=n+1}^{2n}\frac{j^1_k-i^1_{k-n}}{N^k}\Big)^2 + \Big( \sum_{k=1}^n\frac{j_k^2-i_k^2}{N^{k}} + \sum_{k=n+1}^{2n}\frac{j^2_k-i^2_{k-n}}{N^k}\Big)^2} \\
        =:& \sqrt{(I_1+J_1)^2+(I_2+J_2)^2}.
    \end{align*}
    It is easy to see that: (1) $|I_1|\geq N^{-n}$; (2) $|J_1|\leq N^{-n}(1-N^{-n})$. Moreover, if the inequality in (2) is an equality, then either $j_k^1=N-1,i_{k-n}^1=0$ for $n+1\leq k\leq 2n$, or $j_k^1=0,i_{k-n}^1=N-1$ for $n+1\leq k\leq 2n$. Since $0\leq j_k^1\leq N-1$ for all $k$, $I_1$ and $J_1$ must have the same sign if the inequality in (2) is an equality. Thus if the inequalities in (1) and (2) are both equalities, then $|I_1+J_1|\geq |I_1|=N^{-n}$. But if one of them is a strict inequality, then we have $|I_1+J_1|\geq |I_1|-|J_1|>N^{-2n}$. To conclude, we always have $|I_1+J_1|>N^{-2n}$. Similarly, $|I_2+J_2|>N^{-2n}$. Therefore,
    \[
        |\vp_{j_1\cdots j_{2n}}((0,0)) - \vp_{\i\i}((0,0))| > \sqrt{N^{-4n}+N^{-4n}} = \sqrt{2}N^{-2n},
    \]
    as desired.

    The second statement is a direct consequence of the first one. To see this, fix any word $\j=\j_1 \j_2\cdots\j_{k-1}\in\D^{(k-1)n}\setminus\{\i^{(k-1)}\}$, where $\j_p\in \D^n$ for all $1\leq p\leq k-1$. Let $q=\min\{1\leq p\leq k-1: \j_p\neq\i\}$. Then
    \begin{align*}
        \vp_{\j}(F) \cap \vp_{\i^k}(F) =  \vp_{\i^{q-1}}(\vp_{\j_q\cdots\j_{k-1}}(F) \cap \vp_{\i^{k-q+1}}(F))
        \subset \vp_{\i^{q-1}}(\vp_{\j_q}(F)\cap\vp_{\i\i}(F)) = \varnothing,
    \end{align*}
    with the convention that $\vp_{\i^0}=\id$ (the identity map). 
\end{proof}

\section{Essential cut vertices of Hata graphs}

To examine the value $\chi(\Gamma_n)$, we will transform the problem to the detection of the existence of ``essential cut vertices'' defined as follows.

\begin{definition}\label{de:goodcp}
    Let $n\geq 2$ and let $\bm{i}=i_1\cdots i_n\in\D^n$ be a cut vertex of $\Gamma_n$. We call $\bm{i}$ \emph{essential} if there are $i,j\in\D\setminus\{i_1\}$ such that $i\D^{n-1}$ and $j\D^{n-1}$ belong to different connected components of $\Gamma_n-\{\i\}$. As a matter of convenience, we also call every cut vertex of $\Gamma_1$ essential.
\end{definition}

We remark that  by Lemma~\ref{lem:equivalent}, $\bm{i}=i_1\cdots i_n\in\D^n$ is essential if and only if there are $i,j\in\D\setminus\{i_1\}$ such that $\vp_i(F)$ and $\vp_j(F)$ belong to different components of $\bigcup_{w\in\D\setminus\{i_1\}}\vp_w(F)$.

\begin{example}
    The $2$-nd Hata graph associated with the connected GSC in Figure~\ref{fig:exa_goodcp} has exactly five cut vertices
    \[
        (0,2)(0,2),\, (0,2)(2,2),\, (1,0)(1,0),\, (2,2)(0,2),\, (2,2)(2,2).
    \]
    It is easy to see that $(1,0)(1,0)$ is an essential cut vertice while others are not.
    \begin{figure}[htbp]
        \centering
        \includegraphics[height=4cm]{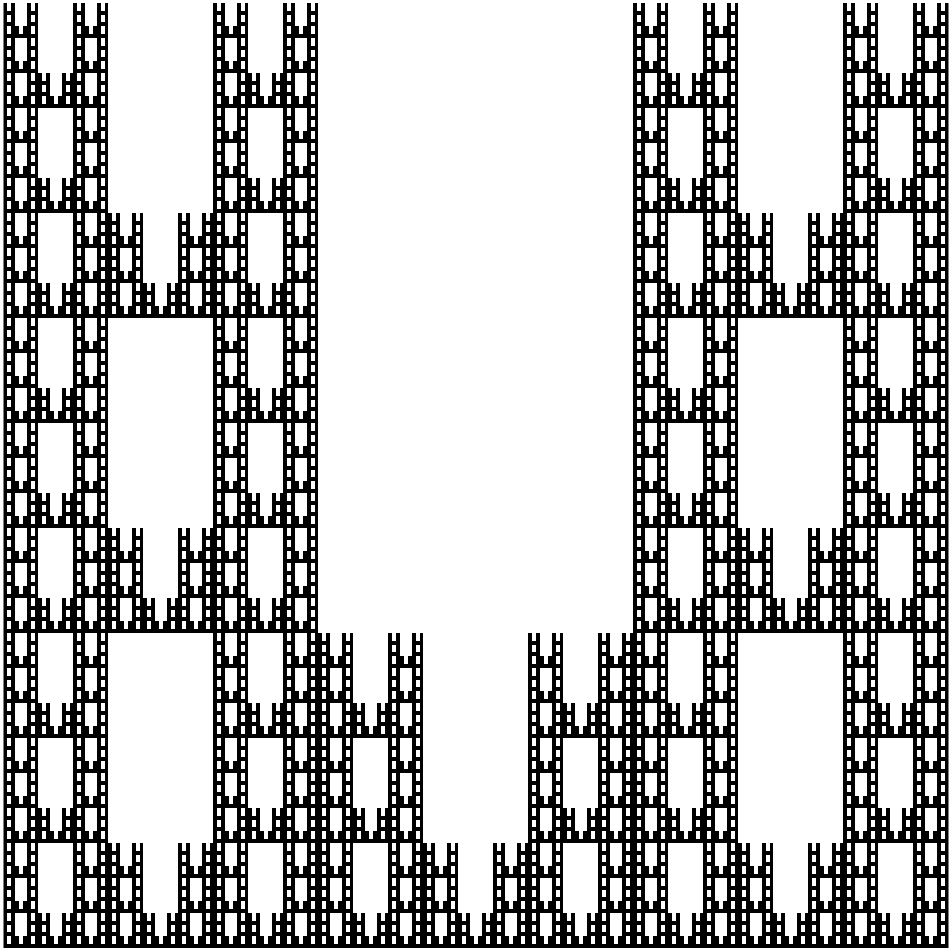} \quad
        \includegraphics[height=4cm]{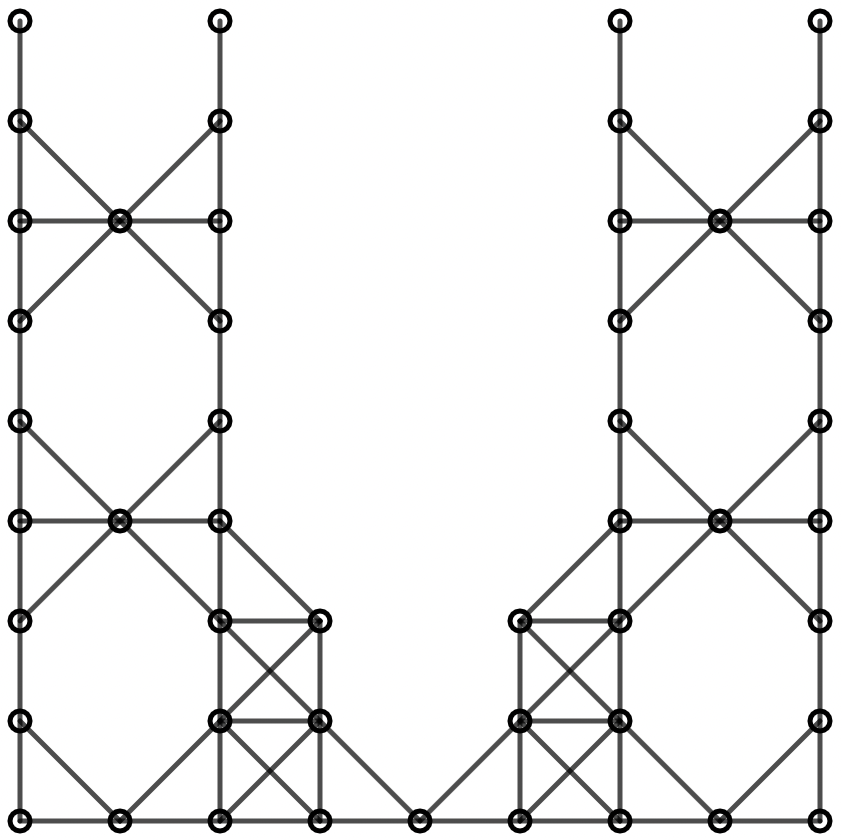}
        \caption{A connected GSC, where $N=3$ and $\D=\{0,1,2\}^2\setminus\{(1,1),(1,2)\}$.}
        \label{fig:exa_goodcp}
    \end{figure}
\end{example}

\begin{lemma}\label{lem:ingoodthenin1}
    Let $n\geq 1$. If $\i\in\D^n$ is an essential cut vertex of $\Gamma_n$, then $\i|_k$ is an essential cut vertex of $\Gamma_k$ for all $1\leq k\leq n$.
\end{lemma}
\begin{proof}
   This is a direct consequence of Lemma~\ref{lem:nonfragilelast}.
\end{proof}

When $F$ is non-fragile, it has been proved in~\cite[Theorem 1.11]{DLRWX21} that $F$ has cut points if and only if $\chi(\Gamma_n)\geq |\D|^{n-1}$ for all $n\geq 2$. The reason we turn to finding essential cut vertices is based on the following simple observation.

\begin{lemma}\label{lem:pra1}
    For $n\geq 2$, $\chi(\Gamma_n)\geq|\D|^{n-1}$ if and only if $\Gamma_n$ has an essential cut vertex. As a result, the GSC $F$ has cut points if and only if $\Gamma_n$ has essential cut vertices for all $n\geq 2$.
\end{lemma}
\begin{proof}
    Suppose $\chi(\Gamma_n)\geq |\D|^{n-1}$ and let $\i=i_1\cdots i_n$ be the vertex of $\Gamma_n$ achieving $\chi(\Gamma_n)$. By~\cite[Lemma 4.5]{DLRWX21}, there are $i,j\in\D\setminus\{i_1\}$ such that $i\D^{n-1}$ and $j\D^{n-1}$ belong to different connected components of $\Gamma_n-\{\i\}$. So $\i$ is an essential cut vertex of $\Gamma_n$.

    Conversely, suppose that $\Gamma_n$ has an essential cut vertex $\j=j_1\cdots j_n$. By definition, there are two digits $i',j'\in\D\setminus\{j_1\}$ such that $i'\D^{n-1}$ and $
    j'\D^{n-1}$ belong to different connected components of $\Gamma_n-\{\j\}$. Therefore, $\chi(\Gamma_n)\geq \min\{|i'\D^{n-1}|, |j'\D^{n-1}|\} = |\D|^{n-1}$.
\end{proof}


The following result provides us with a convenient suffient condition for a connected GSC to have cut points. In many circumtances, there is a cut vertex $i$ of $\Gamma_1$ satisfying the required conditions and hence the GSC has cut points.

\begin{proposition}\label{prop:twogoodmeanscutpt}
    Let $n\geq 1$ and let $\i$ be a cut vertex of $\Gamma_n$. Suppose that there are $\Lambda,\Lambda'\subset\D^n$ satisfying the following conditions:
    \begin{enumerate}
        \item $\Lambda\cup\Lambda'=\D^n\setminus\{\i\}$ and $\Lambda\cap\Lambda'=\varnothing$;
        \item Writing $X_1=\bigcup_{\j\in\Lambda}\vp_{\j}(F)$ and $X'_1=\bigcup_{\j\in\Lambda'}\vp_{\j}(F)$, we have $X_1\cap X'_1=\varnothing$;
        \item Both of $\vp_{\i}(X_1)$ and $\vp_{\i}(X'_1)$ cannot intersect $X_1$ and $X'_1$ simultaneously;
        \item There are $i,j\in\D$ such that $\vp_i(F)\subset X_1$ and $\vp_j(F)\subset X'_1$.
    \end{enumerate}
    Then $\i^k$ is an essential cut vertice of $\Gamma_{kn}$ for all $k\geq 1$.
\end{proposition}

Note that by Lemma~\ref{lem:equivalent}, the fourth condition above indicates that $\i$ is essential.

\begin{proof}
    For $k\geq 1$, let $E_k=\bigcup_{\j\in\D^{kn}\setminus\{\i^k\}}\vp_{\j}(F)$, $X_k=\vp_{\i^{k-1}}(X_1)$ and $X'_k=\vp_{\i^{k-1}}(X'_1)$. By the conditions (1) and (2), $E_1=X_1\cup X'_1$. Note that
    \begin{align*}
        E_k &= \Big( \bigcup_{\j\in\D^{kn}\setminus\{\i^k\},\i^{k-1}\nprec\j}\vp_{\j}(F) \Big) \cup \Big( \bigcup_{\j\in\D^{kn}\setminus\{\i^k\},\i^{k-1}\prec\j}\vp_{\j}(F) \Big) \\
        &= \Big( \bigcup_{\j\in\D^{(k-1)n}\setminus\{\i^{k-1}\}}\vp_{\j}(F) \Big) \cup \vp_{\i^{k-1}}\Big( \bigcup_{\j\in\D^n\setminus\{\i\}}\vp_{\j}(F) \Big) \\
        &= E_{k-1} \cup \vp_{\i^{k-1}}(E_1) = E_{k-1} \cup (X_k\cup X'_k).
    \end{align*}
    By an induction argument, it is easy to see that $E_k=\bigcup_{t=1}^k (X_t\cup X'_t)$ for $k\geq 1$. Write $\mathscr{C}_k=\bigcup_{t=1}^k X_t$ and $\mathscr{C}'_k=\bigcup_{t=1}^{k} X'_t$. It is clear that $\mathscr{C}_k, \mathscr{C}'_k$ are both compact sets and $E_k=\C_k\cup\C'_k$.
    We also have
    \begin{equation}\label{eq:xkandxkprime}
        X_{k+1}\cap X'_{k+1}=\vp_{\i^k}(X_1\cap X'_1)=\varnothing, \quad \forall k\geq 1.
    \end{equation}
    Moreover, we have by Lemma~\ref{lem:iiawayfromj} that for all $k\geq 2$ and $q\geq k+1$,
    \begin{align}
        (X_{q}\cup X'_{q}) \cap E_{k-1} &= (X_{q} \cup X'_{q})\cap \bigcup_{\j\in\D^{(k-1)n}\setminus\{\i^{k-1}\}}\vp_{\j}(F) \notag\\
        &\subset \vp_{\i^{k}}(F) \cap \bigcup_{\j\in\D^{(k-1)n}\setminus\{\i^{k-1}\}}\vp_{\j}(F) = \varnothing. \label{eq:xkcapek}
    \end{align}

    Without loss of generality, the condition (3) can be divided into three cases:
    \begin{enumerate}
        \item[(3-1)] $\vp_{\i}(X_1)\cap X_1=\varnothing$ and  $\vp_{\i}(X'_1)\cap X_1=\varnothing$;
        \item[(3-2)] $\vp_{\i}(X_1)\cap X_1'=\varnothing$ and $\vp_{\i}(X'_1)\cap X_1=\varnothing$;
        \item[(3-3)] $\vp_{\i}(X_1)\cap X_1=\varnothing$ and $\vp_{\i}(X'_1)\cap X_1'=\varnothing$.
    \end{enumerate}

    \textbf{Case 1}. We have (3-1), i.e., $X_1\cap X_2=\varnothing$ and $X_1\cap X'_2=\varnothing$. In this case,
    \[
        X_1 \cap (X_2\cup \C'_2) = (X_1\cap X_2) \cup (X_1\cap X'_1) \cup (X_1\cap X'_2) = \varnothing.
    \]
    It then follows from~\eqref{eq:xkcapek} that for $k\geq 3$,
    \begin{align*}
        X_1 \cap \Big( \C'_k \cup \bigcup_{t=2}^k X_t \Big) = X_1 \cap \Big( \bigcup_{t=3}^k X_t\cup X'_t \Big) \subset E_1 \cap \Big( \bigcup_{t=3}^k X_t\cup X'_t \Big) = \varnothing.
    \end{align*}
    Combining this with $E_k=X_1 \cup \big(\C'_k\cup \bigcup_{t=2}^k X_t \big)$ and the condition (4), $\i^{k}$ is an essential cut vertex of $\Gamma_{kn}$.

    \textbf{Case 2}. We have (3-2), i.e., $X_1\cap X'_2=\varnothing$ and $X'_1\cap X_2=\varnothing$. We will show by induction that $\mathscr{C}_k\cap\mathscr{C}'_k=\varnothing$ for all $k\geq 2$. Then combining this with $E_k=\C_k\cup\C'_k$ and the condition (4), $\i^k$ is an essential cut vertex of $\Gamma_{kn}$.

    Combining the condition (1) and (3-2), we have $\C_1\cap\C'_1=\varnothing$ and $\C_2\cap\C'_2=\varnothing$. Suppose we have shown that  $\mathscr{C}_t\cap\mathscr{C}'_t=\varnothing$ for $1\leq t\leq k$. Then from~\eqref{eq:xkandxkprime},
    \[
        \C_{k+1} \cap \C'_{k+1} = (X_{k+1}\cap \C'_k) \cup ( X'_{k+1}\cap\C_k ).
    \]
    Note that
    \[
        X_{k+1}\cap \C'_k = X_{k+1} \cap (\C'_{k-1}\cup X'_k) = (X_{k+1}\cap\C'_{k-1}) \cup (X_{k+1}\cap X'_k),
    \]
    and $X_{k+1}\cap X'_k=\vp_{\i^{k-1}}(X_2 \cap X'_1)=\varnothing$. It then follows from~\eqref{eq:xkcapek} that
    \[
        X_{k+1}\cap \C'_k = X_{k+1}\cap\C'_{k-1} \subset X_{k+1}\cap E_{k-1} = \varnothing.
    \]
    Similarly, $X'_{k+1}\cap\C_k=\varnothing$. This completes the induction process.

    \textbf{Case 3}. We have (3-3), i.e., $X_1\cap X_2=\varnothing$ and $X'_1\cap X'_2=\varnothing$. In this case, it follows that
    \begin{equation}\label{eq:xkandxkprime2}
        X_{2k-1}\cap X_{2k} = \vp_{\i^{2k-2}}(X_1\cap X_2)=\varnothing \quad\text{and}\quad X'_{2k-1}\cap X'_{2k} = \vp_{\i^{2k-2}}(X'_1\cap X'_2)=\varnothing.
    \end{equation}
    For $k\geq 1$, let $Y_k=X_{2k-1}\cup X'_{2k}$ and $Y'_k=X'_{2k-1}\cup X_{2k}$. Then $E_{2k}=\bigcup_{t=1}^k (Y_t\cup Y'_t)$. It is also not hard to verify the following facts.
    \begin{itemize}
        \item For $k\geq 1$, $Y_{k+1}=\vp_{\i^2}(Y_k)$ and $Y'_{k+1}=\vp_{\i^2}(Y'_k)$;
        \item For $k\geq 1$, $Y_k\cap Y'_k=\varnothing$. In fact, since
        \[
            Y_k \cap Y'_k = (X_{2k-1}\cup X'_{2k}) \cap (X'_{2k-1}\cup X_{2k}),
        \]
        the emptiness follows directly from~\eqref{eq:xkandxkprime} and~\eqref{eq:xkandxkprime2};
        \item $Y_1\cap Y'_2=\varnothing$ and $Y'_1\cap Y_2=\varnothing$. To see this, first note that
        \begin{align*}
            Y_1 \cap Y'_2 &= (X_1\cup X'_2) \cap (X'_3\cup X_4) \\
            &= (X_1 \cap X'_3) \cup (X_1\cap X_4) \cup (X'_2\cap X'_3) \cup (X'_2\cap X_4).
        \end{align*}
        Since $X_1\subset E_1$, we see by~\eqref{eq:xkcapek} that $X_1\cap X'_3=\varnothing$. Similarly, $X_1\cap X_4$ and $X'_2\cap X_4$ are both empty. Finally, $X'_2\cap X'_3=\vp_{\i}(X'_1\cap X'_2)=\varnothing$. So $Y_1 \cap Y'_2=\varnothing$. One can show that $Y'_1\cap Y_2=\varnothing$ by an analogous argument.
    \end{itemize}
    In conclusion, this case is essentially the same as Case 2.
\end{proof}

\begin{remark}\label{rem:fixedptcutpt}
    Note that in the above proof, we actually show that $X_1$ and $X'_1$ belong to different connected components of $\bigcup_{\j\in\D^{kn}\setminus\{\i^k\}}\vp_{\j}(F)$ for all $k\geq 1$. Write $x$ to be the fixed point of the map $\vp_{\i}$. Combining Lemma~\ref{lem:iiawayfromj} with~\cite[Theorem 4.2]{DLRWX21}, it is easy to see that $x$ is a cut point of $F$.
\end{remark}

\begin{corollary}\label{cor:twoandgoodmeanscutpt}
    Let $n\geq 1$ and let $\i$ be an essential cut vertex of $\Gamma_n$ such that $\Gamma_n-\{\i\}$ has exactly two connected components. If $\i\i$ is also an essential cut vertex of $\Gamma_{2n}$, then $\i^k$ is an essential cut vertex of $\Gamma_{kn}$ for all $k\geq 1$.
\end{corollary}
\begin{proof}
    Let $V, V'$ be the vertex sets of the two connected components of $\Gamma_n-\{\i\}$, respectively. In particular, $V\cup V'=\D^n\setminus\{\i\}$. Let $X_1=\bigcup_{\j\in V}\vp_{\j}(F)$ and $X'_1=\bigcup_{\j\in V'}\vp_{\j}(F)$. Then
    \begin{align}
        \bigcup_{\j\in\D^{2n}\setminus\{\i\i\}} \vp_{\j}(F) &= \Big( \bigcup_{\j\in\D^{2n}\setminus\{\i\i\},\i\nprec\j}\vp_{\j}(F) \Big) \cup \Big( \bigcup_{\j\in\D^{2n}\setminus\{\i\i\},\i\prec\j}\vp_{\j}(F) \Big) \notag \\
        &= \Big( \bigcup_{\j\in\D^{n}\setminus\{\i\}}\vp_{\j}(F) \Big) \cup \vp_{\i}\Big( \bigcup_{\j\in\D^n\setminus\{\i\}}\vp_{\j}(F) \Big) \notag \\
        &= (X_1\cup X'_1) \cup (\vp_{\i}(X_1) \cup \vp_{\i}(X'_1)). \label{eq:cor4-7}
    \end{align}
    Note that $X_1$ and $X'_1$ are both connected. Thus, from \eqref{eq:cor4-7} and the fact that $\i\i$ is an essential cut vertex of $\Gamma_{2n}$, $X_1$ and $X'_1$ belong to different connected components of $\bigcup_{\j\in\D^{2n}\setminus\{\i\i\}}\vp_{\j}(F)$. Using \eqref{eq:cor4-7} again, both of $\vp_{\i}(X_1)$ and $\vp_{\i}(X'_1)$ cannot intersect $X_1$ and $X'_1$ simultaneously. Now the corollary follows directly from Proposition~\ref{prop:twogoodmeanscutpt}.
\end{proof}

\begin{proposition}\label{prop:stillgood}
    Let $n\geq 3$. If $\i=i_1\cdots i_n\in\D^n$ is an essential cut vertex of $\Gamma_n$, then either $i_2\cdots i_n$ is an essential cut vertex of $\Gamma_{n-1}$, or $i_3\cdots i_n$ is an essential cut vertex of $\Gamma_{n-2}$.
\end{proposition}

Since the proof of Proposition~\ref{prop:stillgood} involves a rather technical case-by-case discussion, we decide to present it in Section 6 so readers can move on without being overwhelmed by tedious details. One can take this proposition for granted at this moment and come back to the proof of it later.

\begin{remark}
    In Proposition~\ref{prop:stillgood}, $i_2\cdots i_n$ is not necessarily essential. For example, consider the GSC as in Figure~\ref{fig:exa_iigood}. It is not hard to see that the following facts hold.
    \begin{enumerate}
        \item The GSC is non-fragile and connected;
        \item $(0,0)(0,0)$ is an essential cut vertex of $\Gamma_2$;
        \item $(0,2)(4,0)(0,0)(0,0)$ is an essential cut vertex of $\Gamma_4$;
        \item $(4,0)(0,0)(0,0)$ is not an essential cut vertex of $\Gamma_3$.
    \end{enumerate}
    \begin{figure}[htbp]
        \centering
        \includegraphics[width=4.1cm]{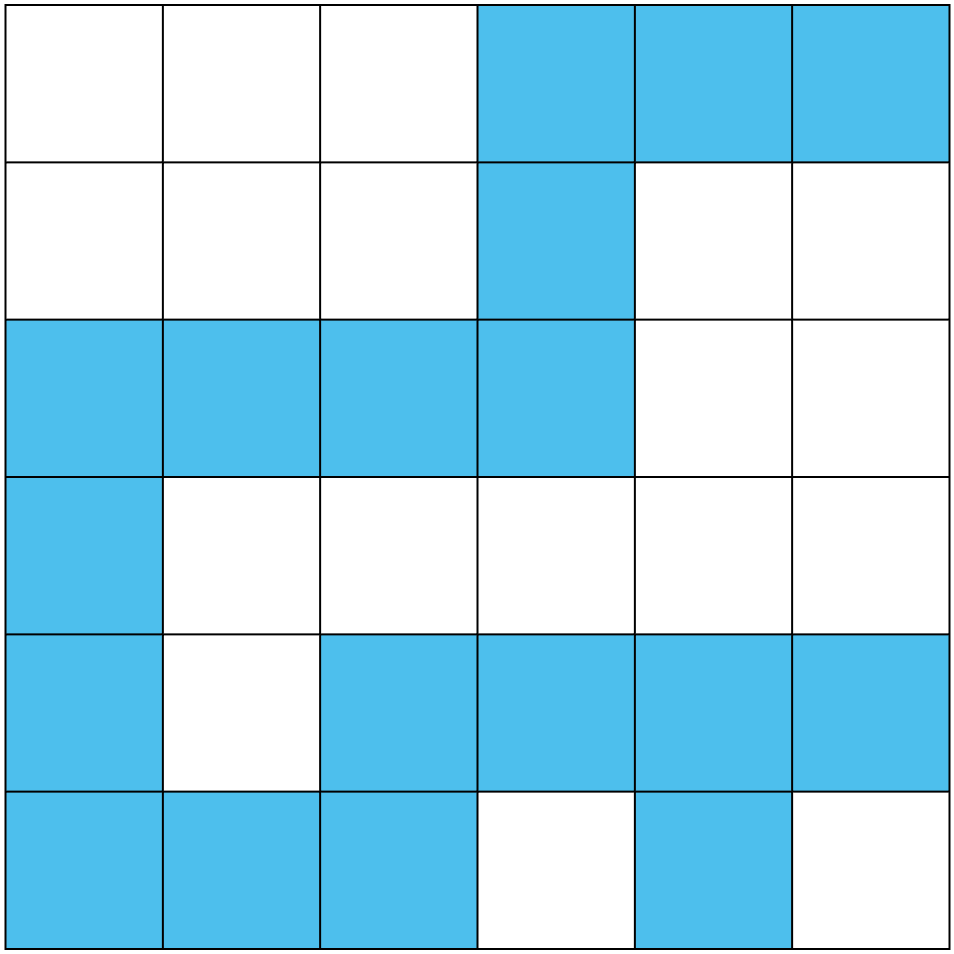} \quad
        \includegraphics[width=4.1cm]{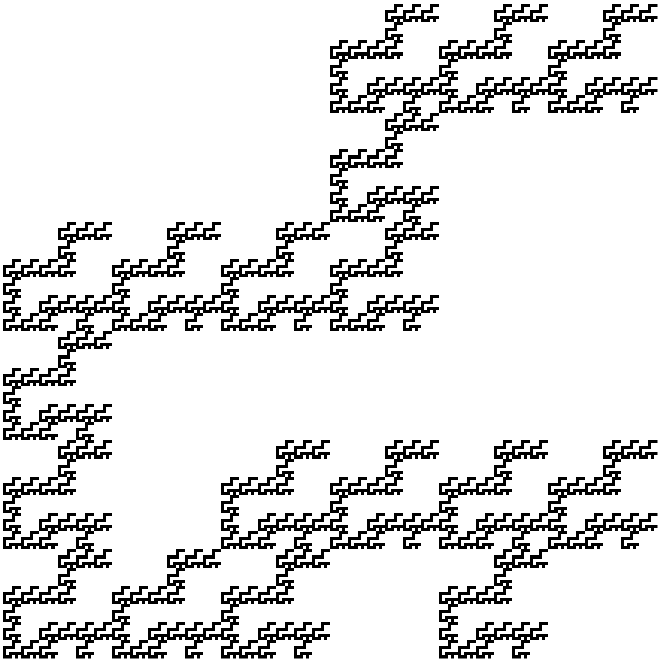} \quad
        \includegraphics[width=4.1cm]{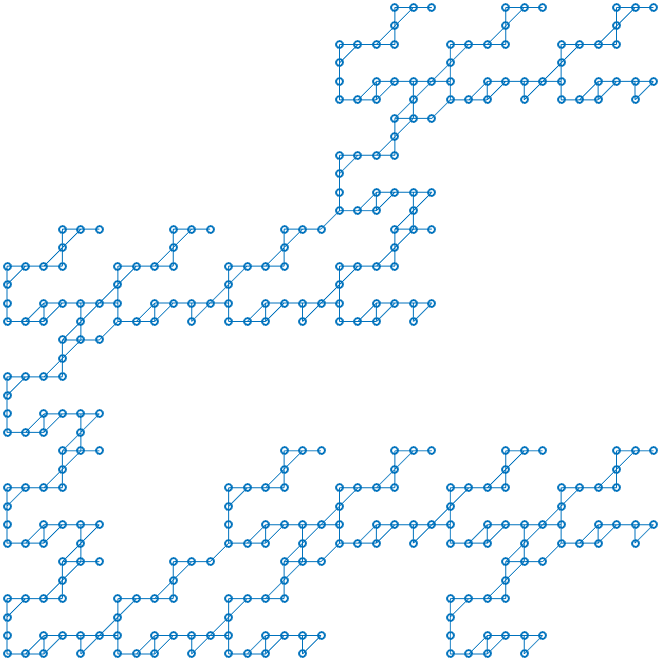}
        \caption{From left to right: the initial pattern, the GSC and the $2$-nd Hata graph}
        \label{fig:exa_iigood}
    \end{figure}
\end{remark}

\section{Detecting essential cut vertices in finitely many steps}
\subsection{Positions of neighbor cells}
One of the key ingredient in the proof of Theorem~\ref{thm:goodcp} is the grid structure in the construction of GSCs. More precisely, every basic square to be concerned with is surrounded by at most $8$ squares of the same side length, and we will label these positions as in Figure~\ref{fig:8positions}. For convenience, we write $\P:=\{\uparrow,\downarrow,\leftarrow,\rightarrow,\swarrow,\searrow,\nwarrow,\nearrow\}$ as the collection of these labels. The following definition is rather a way of notation.

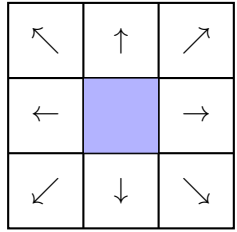
\begin{figure}[htbp]
    \centering
        \begin{tikzpicture}
            \draw[thick] (0,0) rectangle (3,3);
            \draw[thick] (0,1) to (3,1);
            \draw[thick] (0,2) to (3,2);
            \draw[thick] (1,0) to (1,3);
            \draw[thick] (2,0) to (2,3);
            \filldraw[draw=black,fill=blue!30] (1,1) rectangle (2,2);
            \node at(0.5,0.5) {$\swarrow$};
            \node at(1.5,0.5) {$\downarrow$};
            \node at(2.5,0.5) {$\searrow$};
            \node at(0.5,1.5) {$\leftarrow$};
            \node at(2.5,1.5) {$\rightarrow$};
            \node at(0.5,2.5) {$\nwarrow$};
            \node at(1.5,2.5) {$\uparrow$};
            \node at(2.5,2.5) {$\nearrow$};
        \end{tikzpicture}
        \caption{The eight neighbor squares and their labels}
        \label{fig:8positions}
\end{figure}


\begin{definition}\label{de:position}
    For every $\i\in\D^*$ and $t\in\P$, denote $\bm{i}(t)$ to be the element in $\D^{|\i|}$, if there is such one, satisfying the following conditions:
    \begin{enumerate}
        \item $\vp_{\i(t)}([0,1]^2)$ is the square (of the same size) lying exactly in the position $t$ adjacent to $\vp_{\i}([0,1]^2)$;
        \item $\vp_{\i(t)}(F)\cap\vp_{\i}(F)\neq\varnothing$.
    \end{enumerate}
\end{definition}

The second condition guarantees that $\vp_{\i(t)}(F)$ is a ``truly'' neighbor cell of $\vp_{\i}(F)$. The following is a simple geometric observation.

\begin{lemma}\label{lem:position}
    Let $\i,\j\in\D^*$ and let $P\subset\P$ be such that  $\i(t), \j(t)$ are well defined for all $t\in P$. Then for any $A\subset F$,
    \[
        \Big( \bigcup_{t\in P}\vp_{\i(t)}(F) \Big) \cap \vp_{\i}(A) = \varnothing \Longleftrightarrow \Big( \bigcup_{t\in P}\vp_{\j(t)}(F) \Big) \cap \vp_{\j}(A) = \varnothing.
    \]
\end{lemma}
\begin{proof}
    By the self-similarity, for each fixed $t\in P$, $\vp_{\i(t)}(F)\cap\vp_{\i}(A)$ is just a scaled copy of $\vp_{\j(t)}(F)\cap\vp_{\j}(A)$. So one of them is empty if and only if the other is empty. Then the lemma follows immediately.
\end{proof}

\subsection{Proof of Theorem~\ref{thm:goodcp}}
Now let us begin the proof. Let $M$ be a large positive integer that will be specified later and assume that $\chi(\Gamma_M)\geq|\D|^{M-1}$. Then there exists by Lemma~\ref{lem:pra1} an essential cut vertex $\i=i_1\cdots i_M$ of $\Gamma_M$. By definition, we can find $i_*,j_*\in\D\setminus\{i_1\}$ such that $i_*\D^{M-1}$ and $j_*\D^{M-1}$ belong to different connected components of $\Gamma_M-\{\i\}$. Equivalently (again by Lemma~\ref{lem:equivalent}), $\vp_{i_*}(F)$ and $\vp_{j_*}(F)$ belong to different components of $\bigcup_{\j\in\D^M\setminus\{\i\}}\vp_{\j}(F)$, namely $\C_{i_*}$ and $\C_{j_*}$. For $1\leq n\leq M$, write
\[
    V_n = \{\j\in\D^n: \vp_{\j}(F)\subset\C_{i_*}\}, \quad \C_n = \bigcup_{\j\in V_n} \vp_{\j}(F)
\]
and
\[
    V'_n = \D^n\setminus(\{\i|_n\}\cup V_n), \quad \C'_n = \bigcup_{\j\in V'_n} \vp_{\j}(F).
\]
Note that both of $\C_n,\C'_n$ are finite unions of level-$n$ cells and $\C_n\cap\C'_n=\varnothing$. Furthermore, $\{\C_n\}_{n=1}^M$ and $\{\C'_n\}_{n=1}^M$ are both increasing sequences.

Recall that $\P=\{\uparrow,\downarrow,\leftarrow,\rightarrow,\swarrow,\searrow,\nwarrow,\nearrow\}$. For $1\leq n\leq M$, let
\[
    \P_n = \{t\in\P: \vp_{\i|_n(t)}(F)\subset \C_n\} \quad\text{and}\quad \P'_n = \{t\in\P: \vp_{\i|_n(t)}(F)\subset \C'_n\}.
\]
That is to say, $\P_n$ (resp. $\P'_n$) records positions of level-$n$ cells that is ``adjacent'' to $\vp_{\i|_n}(F)$ and contained in $\C_n$ (resp. $\C'_n$). Clearly, $\P_n\cap\P'_n=\varnothing$. Since $\i|_n$ is a cut vertex of $\Gamma_n$, it follows from Lemma~\ref{lem:graphcon} that $\P_n$ and $\P'_n$ are both non-empty. In particular, by Definition~\ref{de:position}, we see that $\C_n\cap \vp_{\i|_n}(F)\neq\varnothing$ and $\C'_n\cap \vp_{\i|_n}(F)\neq\varnothing$.

By Proposition~\ref{prop:stillgood}, there is a sequence $1\leq k_1<k_2<\cdots\leq M-2$ such that for all $p$:
\begin{enumerate}
    \item $i_{k_p}\cdots i_M$ is an essential cut vertex of $\Gamma_{M-k_p+1}$;
    \item $k_p<k_{p+1}\leq k_p+2$.
\end{enumerate}
Note that $\P$ is a finite set. So taking $M$ large enough in the beginning, we can find $n_1<n_2$ such that
\[
    \{\P_{k_{n_1}}, \P'_{k_{n_1}}\} = \{\P_{k_{n_2}}, \P'_{k_{n_2}}\}.
\]
For example, since $\P$ has at most $2^{-1}\cdot 3^{|\P|}=2^{-1}\cdot 3^8$ distinct unordered pairs of disjoint subsets, taking $M=3^8+3$ will suffice. By Lemma~\ref{lem:ingoodthenin1}, $i_{k_{n_1}}\cdots i_{k_{n_2}}$ is also essential. To avoid complicated subscripts, we may replace $k_{n_1}, k_{n_2}$ with $n_1,n_2$, respectively, and write
\[
    \omega := i_{n_1+1}\cdots i_{n_2}.
\]
Furthermore, let
\begin{equation}\label{eq:pomegaprime}
    \P_\omega = \{t\in\P: \vp_{\i|_{n_1}}(\vp_{\omega(t)}(F)) \subset \C_{n_2}\} \quad\text{and}\quad \P'_\omega = \{t\in\P: \vp_{\i|_{n_1}}(\vp_{\omega(t)}(F)) \subset \C'_{n_2} \}.
\end{equation}

\begin{lemma}\label{lem:qsubsetp}
    $\P_\omega\subset \P_{n_2}$, $\P'_\omega\subset \P'_{n_2}$.
\end{lemma}
\begin{proof} 
    If $t\in \P_\omega$, then
    \[
      \vp_{\i|_{n_1}\omega(t)}(F) \cap \vp_{\i|_{n_2}}(F) = \vp_{i|_{n_1}} (\vp_{\omega(t)}(F) \cap \vp_\omega(F))\not=\varnothing,
    \]
    and $\vp_{\i|_{n_1}\omega(t)}([0,1]^2)$ is the square (of the same size) lying exactly in the position $t$ adjacent to $\vp_{\i|_{n_2}}([0,1]^2)$. Thus $\i|_{n_2}(t)=\i|_{n_1}\omega(t)$ so that $t\in \P_{n_2}$. Hence $\P_\omega\subset \P_{n_2}$. Similarly, $\P'_\omega\subset \P'_{n_2}$.
\end{proof}


\begin{lemma}\label{lem:5-3}
    There are $i,j\in \D\setminus \{i_{n_1+1}\}$ with $\vp_{\i|_{n_1}i}(F)\cap\C_{n_1}\neq\varnothing$ and $\vp_{\i|_{n_1}j}(F)\cap\C'_{n_1}\neq\varnothing$.
\end{lemma}
\begin{proof}
    We will prove the lemma by contradiction. Suppose on the contrary that $\vp_{\i|_{n_1+1}}(F)$ is the only level-$(n_1+1)$ cell in $\vp_{\i|_{n_1}}(F)$ which intersects $\C_{n_1}$.

    {\bf Claim}. There exists some level-$n_1$ cell in $\C_{n_1}$ that contains at least two level-$(n_1+1)$ cells intersecting $\vp_{\i|_{n_1}}(F)$.

    Otherwise, for every level-$n_1$ cell $\vp_{\j}(F)\subset\C_{n_1}$ with $\vp_{\j}(F)\cap\vp_{\i|_{n_1}}(F)\neq\varnothing$, we see by Lemma~\ref{lem:casedisfact} that $\vp_{\j}(F) \cap \vp_{\i|_{n_1}}(F)$ is merely a singleton. If there are more than one such cells, then $i_{n_1+1}\in\{(0,0),(N-1,0),(0,N-1),(N-1,N-1)\}$. By Lemma~\ref{lem:fourvertex}, these singletons must be identical (i.e., $\{\vp_{\i|_{n_1}}(\frac{i_{n_1+1}}{N-1})\}$). So in conclusion, $\C_{n_1}\cap\vp_{\i|_{n_1}}(F)$ is just a singleton. Recall that $\C_{n_1}, \C'_{n_1}$ are both finite union of level-$n_1$ cells and they are disjoint. Therefore,
    \[
        \Big( \bigcup_{\j\in V_{n_1}}\vp_{\j}(F) \Big) \cap \Big( \bigcup_{\j\in\D^{n_1}\setminus V_{n_1}}\vp_{\j}(F) \Big) = \C_{n_1} \cap (\C'_{n_1}\cup\vp_{\i|_{n_1}}(F)) = \C_{n_1}\cap\vp_{\i|_{n_1}}(F)
    \]
    is a singleton. By Lemma~\ref{lem:evenfragile}, $F$ is fragile and we obtain a contradiction. This proves the claim.

    Let $\vp_{\eta}(F)\subset\C_{n_1}$ be a level-$n_1$ cell as in the above claim. Recall from our hypothesis in the beginning that $\vp_{\i|_{n_1+1}}(F)$ is the only level-$(n_1+1)$ cell in $\vp_{\i|_{n_1}}(F)$ meeting $\vp_{\eta}(F)$. So it suffices to discuss the following two cases.

    \textbf{Case 1}. There are exactly two level-$(n_1+1)$ cells in $\vp_{\eta}(F)$ which intersects $\vp_{\i|_{n_1+1}}(F)$. Rotating or reflecting if necessary, the first two cases (from left to right) in Figure~\ref{fig:nonempty} illustrate all possibilities. Note that in both cases, we have $(0,0),(N-1,N-1)\in\D$. Moreover, if it is as the second case in Figure~\ref{fig:nonempty}, we also have $(N-1,0),(0,N-1)\in\D$. But it then follows that
    \[
        \vp_{\i|_{n_1}}(\vp_{(0,0)}(F)) \cap \vp_{\eta}(F) \supset \vp_{\i|_{n_1}}(\vp_{(0,0)}(F)) \cap \vp_{\eta}(\vp_{(0,N-1)}(F)) \neq\varnothing
    \]
    and
    \[
        \vp_{\i|_{n_1}}(\vp_{(N-1,0)}(F)) \cap \vp_{\eta}(F) \supset \vp_{\i|_{n_1}}(\vp_{(N-1,0)}(F)) \cap \vp_{\eta}(\vp_{(N-1,N-1)}(F)) \neq\varnothing,
    \]
    since these intersections of two level-$(n_1+1)$ cells are scaled copies of $\vp_{\i|_{n_1}}(F) \cap \vp_{\eta}(F)$.
    Thus there are at least two level-$(n_1+1)$ cells in $\vp_{\i|_{n_1}}(F)$ that intersects $\vp_\eta(F)\subset\C_{n_1}$. This is a contradiction.

    Now let us consider the first case in Figure~\ref{fig:nonempty}. For convenience, write $i_{n_1+1}=(a,0)$. Since $\vp_{\i|_{n_1+1}}(F)$ is the only level-$(n_1+1)$ cell in $\vp_{\i|_{n_1}}(F)$ which intersects $\vp_{\eta}(F)$, it is not hard to see that $(a-1,0),(a+1,0)\notin\D$. Recall that $i_{n_1+1}\cdots i_{n_2}$ is an essential cut vertex. In particular, $i_{n_1+1}$ is a cut vertex of $\Gamma_1$. But this contradicts Lemma~\ref{lem:notacutpt}.

    \begin{figure}[htbp]
        \centering
            \begin{tikzpicture}[scale=0.8]
                \draw[thick] (-4,0) rectangle (-1.5,2.5);
                \draw[thick] (-4,-2.5) rectangle (-1.5,0);
                \draw[thick] (-3.5,-0.5) rectangle (-3,0);
                \draw[thick] (-3.125,-0.125) rectangle (-3,0);
                \draw[thick] (-3,-0.5) rectangle (-2.5,0);
                \draw[thick] (-3,0) rectangle (-2.875,0.125);
                \draw[thick] (-3,0) rectangle (-2.5,0.5);
                \node[font=\fontsize{16}{1}\selectfont] at(-2.75,-1.25) {$\eta$};
                \node[font=\fontsize{16}{1}\selectfont] at(-2.75,1.5) {$\i|_{n_1}$};
                \draw[->,thick,red] (-2.75,0.65) to (-2.75,0.25);
                \node[font=\fontsize{8}{1}\selectfont,red] at(-2.75,0.85) {$\i|_{n_1+1}$};
                \draw[thick] (0,0) rectangle (2.5,2.5);
                \draw[thick] (0,-2.5) rectangle (2.5,0);
                \draw[thick] (0.875,-0.125) rectangle (1,0);
                \draw[thick] (0.5,-0.5) rectangle (1,0);
                \draw[thick] (1,0) rectangle (1.5,0.5);
                \draw[thick] (1,0) rectangle (1.125,0.125);
                \draw[thick] (1.375,0) rectangle (1.5,0.125);
                \draw[thick] (1.5,-0.5) rectangle (2,0);
                \draw[thick] (1.5,-0.125) rectangle (1.625,0);
                \node[font=\fontsize{16}{1}\selectfont] at(1.25,-1.25) {$\eta$};
                \node[font=\fontsize{16}{1}\selectfont] at(1.25,1.5) {$\i|_{n_1}$};
                \draw[->,thick,red] (1.25,0.65) to (1.25,0.25);
                \node[font=\fontsize{8}{1}\selectfont,red] at(1.25,0.85) {$\i|_{n_1+1}$};
                \draw[thick] (4,0) rectangle (6.5,2.5);
                \draw[thick] (4,-2.5) rectangle (6.5,0);
                \draw[thick] (4.5,-0.5) rectangle (5,0);
                \draw[thick] (4.875,-0.125) rectangle (5,0);
                \draw[thick] (5,-0.5) rectangle (5.5,0);
                \draw[thick] (5,0) rectangle (5.5,0.5);
                \draw[thick] (5,0) rectangle (5.125,0.125);
                \draw[thick] (5.375,0) rectangle (5.5,0.125);
                \draw[thick] (5.5,-0.5) rectangle (6,0);
                \draw[thick] (5.5,-0.125) rectangle (5.625,0);
                \node[font=\fontsize{16}{1}\selectfont] at(5.25,-1.25) {$\eta$};
                \node[font=\fontsize{16}{1}\selectfont] at(5.25,1.5) {$\i|_{n_1}$};
                \draw[->,thick,red] (5.25,0.65) to (5.25,0.25);
                \node[font=\fontsize{8}{1}\selectfont,red] at(5.25,0.85) {$\i|_{n_1+1}$};
            \end{tikzpicture}
            \caption{Local structure between $\vp_{\i|_{n_1}}(F)$ and $\vp_\eta(F)$}
            \label{fig:nonempty}
    \end{figure}
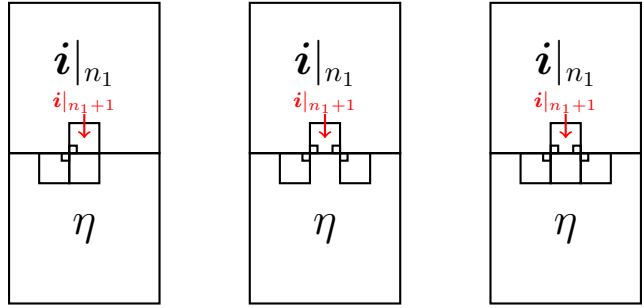

    \textbf{Case 2}. There are three level-$(n_1+1)$ cells in $\vp_{\eta}(F)$ which intersects $\vp_{\i|_{n_1+1}}(F)$. Please see the third figure in Figure~\ref{fig:nonempty} for an illustration. In this case, it is easy to see that
    \[
        \{(0,0),(N-1,0),(0,N-1),(N-1,N-1)\} \subset \D
    \]
    and a contradiction immediately follows as in Case 1.

    The existence of $j$ can be showed similarly. In fact, if we denote by $\mathscr{D}_n$ the union of level-$n$ cells contained in $\C_{j_*}$ (recall this notation in the beginning of this subsection), then applying an analogous argument as above, one can find $j\in \D\setminus \{i_{n_1+1}\}$ such that $\vp_{\i|_{n_1}j}(F) \cap \mathscr{D}_{n_1}\not=\varnothing$. Since $\mathscr{D}_{n_1}\subset\C'_{n_1}$, this completes the proof.
\end{proof}

\begin{remark}\label{rem:bbnonempty}
    Let $i,j$ be two digits as in the above lemma. Since $\C_{n_1+1}$ (resp. $\mathscr{D}_{n_1+1}$) collects all level-$(n_1+1)$ cells contained in the connected component $\C_{i_*}$ (resp. $\C_{j_*}$), we see that $\vp_{\i|_{n_1}i}(F)\subset\C_{n_1+1}$ (resp. $\vp_{\i|_{n_1}j}(F)\subset\mathscr{D}_{n_1+1}\subset\C'_{n_1+1}$). In particular, $i,j$ must be distinct.
\end{remark}

Now we define
\begin{equation}\label{eq:b}
    B = \bigcup\{\vp_{\j}(F): \j\in\D^{n_2}, \i|_{n_1}\prec\j, \vp_{\j}(F)\subset \C_{n_2}\}
\end{equation}
and
\begin{equation}\label{eq:b'}
    B' = \bigcup\{\vp_{\j}(F): \j\in\D^{n_2}, \i|_{n_1}\prec\j, \vp_{\j}(F)\subset \C'_{n_2}\}.
\end{equation}
That is to say, $B$ (resp. $B'$) is the union of those level-$n_2$ cells in the level-$n_1$ cell $\vp_{\i_{n_1}}(F)$ which are contained in $\C_{n_2}$ (resp. $\C'_{n_2}$). Since $\C_{n_2}\cap \C'_{n_2}=\varnothing$, $B\cap B'=\varnothing$. 

\begin{corollary}\label{lem:bbnonempty}
    There is at least one level-$(n_1+1)$ cell contained in $B$ (resp. $B'$). In particular, both of $B$ and $B'$ are non-empty.
\end{corollary}

By the definitions of $B$ and $B'$, it is clear that
\begin{equation}\label{eq:5-4}
    \P_\omega = \{t\in\P: \vp_{\i|_{n_1}}(\vp_{\omega(t)}(F)) \subset B\} \quad\text{and}\quad \P'_\omega = \{t\in\P: \vp_{\i|_{n_1}}(\vp_{\omega(t)}(F)) \subset B' \},
\end{equation}
where $\P_\omega,\P'_\omega$ are as in~\eqref{eq:pomegaprime}. Since $\C_{n_2}\cup\C'_{n_2}=\bigcup_{\j\in\D^{n_2}\setminus\{\i|_{n_2}\}}\vp_{\j}(F)$, we also have
\begin{equation*}\label{eq:ucupu'2}
    \bigcup_{\j\in\D^{|\omega|}\setminus\{\omega\}}\vp_{\j}(F) = \vp_{\i|_{n_1}}^{-1}\Big( \bigcup_{\j\in\D^{|\omega|}\setminus\{\omega\}}\vp_{\i|_{n_1}\j}(F) \Big) = \vp_{\i|_{n_1}}^{-1}(B) \cup \vp_{\i|_{n_1}}^{-1}(B') =: U\cup U'.
\end{equation*}
As a result,
\begin{align*}
    \bigcup_{\j\in\D^{2|\omega|}\setminus\{\omega\omega\}}\vp_{\j}(F) &= \Big( \bigcup_{\j\in\D^{|\omega|}\setminus\{\omega\}}\vp_{\j}(F) \Big) \cup \vp_{\omega}\Big( \bigcup_{\j\in\D^{|\omega|}\setminus\{\omega\}}\vp_{\j}(F) \Big) \\
    &= (U\cup U') \cup \vp_{\omega}(U\cup U') \\
    &= (U\cup\vp_\omega(U)) \cup (U'\cup\vp_\omega(U')).
\end{align*}
Since $B\cap B'=\varnothing$, we have $U\cap U'=\varnothing$ and hence $\vp_{\omega}(U)\cap\vp_{\omega}(U')=\varnothing$.

\begin{lemma}
    The following facts hold.
    \begin{enumerate}
        \item[(1).] There are distinct $i,j\in\D$ such that $\vp_i(F)\subset U$ and $\vp_j(F)\subset U'$;
        \item[(2).] Both $\vp_\omega(U)$ and $\vp_\omega(U')$ cannot intersect $U$ and $U'$ simultaneously.
    \end{enumerate}
\end{lemma}
\begin{proof}
    By Lemma~\ref{lem:bbnonempty}, there are distinct $i,j\in\D$ such that $\vp_{\i|_{n_1}i}(F) \subset B$ and $\vp_{\i|_{n_1}j}(F) \subset B'$. Thus $\vp_i(F)\subset\vp_{\i|_{n_1}}^{-1}(B) = U$ and $\vp_j(F)\subset\vp_{\i|_{n_1}}^{-1}(B') = U'$. This establishes (1).

    For (2), note that by the definition of $U$, Definition~\ref{de:position}, \eqref{eq:5-4} and Lemma~\ref{lem:qsubsetp},
    \begin{align}
        U\cap\vp_\omega(U') &= \vp_{\i|_{n_1}}^{-1}(B) \cap \vp_\omega(U') \notag \\
        &= \Big( \bigcup\{\vp_{\j}(F):\j\in\D^{|\omega|}\setminus\{\omega\},\vp_{\i|_{n_1}\j}(F)\subset B\} \Big) \cap \vp_{\omega}(U') \notag \\
        &= \Big( \bigcup\{\vp_{\omega(t)}(F): t\in\P, \vp_{\i|_{n_1}}(\vp_{\omega(t)}(F))\subset B\} \Big) \cap \vp_{\omega}(U') \notag \\
        &= \Big( \bigcup_{t\in\P_\omega}\vp_{\omega(t)}(F) \Big)\cap \vp_{\omega}(U') \subset \Big( \bigcup_{t\in \P_{n_2}}\vp_{\omega(t)}(F) \Big) \cap \vp_{\omega}(U').\label{eq:ucapvpu1}
    \end{align}
    Similarly,
    \begin{equation}\label{eq:ucapvpu2}
        U'\cap\vp_\omega(U') \subset \Big( \bigcup_{t\in \P'_{n_2}}\vp_{\omega(t)}(F) \Big) \cap \vp_{\omega}(U').
    \end{equation}
    Also note that $\vp_{\i|_{n_1}}(U')=B'\subset \C'_{n_2}$.
    Since $\{\C_n\}_{n=1}^M$ is increasing, we have
    \begin{equation}\label{eq:5-7}
        \Big( \bigcup_{t\in \P_{n_1}}\vp_{\i|_{n_1}(t)}(F) \Big) \cap \vp_{\i|_{n_1}}(U') \subset \C_{n_1} \cap \C'_{n_2} \subset \C_{n_2}\cap\C'_{n_2} = \varnothing.
    \end{equation}

    Recall that $\{\P_{n_1}, \P'_{n_1}\} = \{\P_{n_2}, \P'_{n_2}\}$.
    \begin{enumerate}[(I)]
        \item $\P_{n_1}=\P_{n_2}$. In this case,
        combining Lemma~\ref{lem:position} with~\eqref{eq:ucapvpu1} and~\eqref{eq:5-7},
        \[
            U\cap\vp_{\omega}(U') \subset \Big( \bigcup_{t\in \P_{n_1}}\vp_{\omega(t)}(F) \Big) \cap \vp_{\omega}(U') = \varnothing.
        \]
        \item $\P_{n_1}=\P'_{n_2}$. In this case,
        combining Lemma~\ref{lem:position}  with~\eqref{eq:ucapvpu2} and~\eqref{eq:5-7},
        \[
            U'\cap\vp_{\omega}(U') \subset \Big( \bigcup_{t\in \P_{n_1}}\vp_{\omega(t)}(F) \Big) \cap \vp_{\omega}(U') = \varnothing.
        \]
    \end{enumerate}
    Similarly, $\vp_\omega(U)$ cannot intersect $U$ and $U'$ simultaneously. This establishes (2) and hence completes the proof.
\end{proof}

\begin{proof}[Proof of Theorem~\ref{thm:goodcp}]
    Note that in the above proof, we actually show that the conditions in Proposition~\ref{prop:twogoodmeanscutpt} are fulfilled with $\i=\omega$, $X_1=U$ and $X'_1=U'$. As a consequence, $\omega^k$ is an essential cut vertex of $\Gamma^{k|\omega|}$ for all $k\geq 1$. By Lemma~\ref{lem:ingoodthenin1}, $\Gamma_n$ has essential cut vertices for all $n\geq 1$. Recalling Lemma~\ref{lem:pra1}, the GSC $F$ contains cut points.
\end{proof}

By Remark~\ref{rem:fixedptcutpt}, if a non-fragile connected GSC has cut points, then we can even find a cut point which is the fixed point of $\vp_{\i}$ for some $\i\in \D^*$.

\section{Proof of Proposition~\ref{prop:stillgood}}

To show Proposition~\ref{prop:stillgood}, suppose on the contrary that $i_2\cdots i_n$ is not an essential cut vertex of $\Gamma_{n-1}$. We will prove that this either leads to a contradiction (as in Case 1, Subcase 2.1 and Case 3 later) or that $i_3\cdots i_n$ is an essential cut vertex of $\Gamma_{n-2}$ (as in Subcase 2.2 later).

Our hypothesis implies that $\{i\D^{n-2}:i\in\D\setminus\{i_2\}\}$ lies in one connected component of $\Gamma_{n-1}-\{i_2\cdots i_n\}$. By Lemma~\ref{lem:equivalent} again, the set
\begin{equation}\label{eq:notationofe}
    E:=\bigcup_{i\in\D\setminus\{i_2\}} \vp_{i}(F)
\end{equation}
is a subset of some connected component of $\bigcup_{\j\in\D^{n-1}\setminus\{i_2\cdots i_n\}}\vp_{\j}(F)$. Therefore, $\vp_{i_1}(E)=\bigcup_{j\in \D\setminus \{i_2\}} \vp_{i_1j}(F)$ is contained in exactly one connected component, denoted by $\C$, of $\bigcup_{\j\in\D^n\setminus\{\i\}}\vp_{\j}(F)$.

\begin{lemma}\label{lem:pra2}
    There exists $i_*\in\D\setminus\{i_1\}$ such that $\vp_{i_*}(F)\cap\vp_{i_1i_2}(F)\neq\varnothing$ but $\vp_{i_*}(F)\cap \C=\varnothing$.
\end{lemma}
\begin{proof}
    Let $B_1=\C\cup\vp_{i_1}(F)$ and $A=\vp_{i_1}(F)\setminus \C\neq\varnothing$. Then $B_1=\C\cup A$ is a disjoint union. It follows from $\C\supset\vp_{i_1}(E)$ that $B_1$ is connected and $A\subset \vp_{i_1i_2}(F)$.

    Suppose that $\vp_{j}(F)\cap\vp_{i_1i_2}(F)=\varnothing$ for all $j\in\D\setminus\{i_1\}$. Then
    \begin{equation}\label{eq:vpjcapvpi1-c}
        \vp_{j}(F) \cap A \subset \vp_{j}(F) \cap\vp_{i_1i_2}(F)=\varnothing, \quad \forall j\in\D\setminus\{i_1\}.
    \end{equation}
    Note that $\{B_1\}\cup\{\vp_{j}(F):j\in\D\setminus\{i_1\}\}$ is a family of connected compact sets such that their union $F$ is connected. Since $B_1\setminus A=\C$ is also connected, by \eqref{eq:vpjcapvpi1-c} and Lemma~\ref{lem:connected},
    \[
        G:=\C \cup \Big( \bigcup_{j\in\D\setminus\{i_1\}}\vp_{j}(F) \Big) = (B_1\setminus A) \cup \Big( \bigcup_{j\in\D\setminus\{i_1\}}\vp_{j}(F) \Big)
    \]
    is connected. By the definition of $\C$, $G\subset\bigcup_{\j\in \D^n\setminus \{\i\}} \vp_{\j}(F)$. Thus $\bigcup_{j\in \D\setminus\{i_1\}} \vp_j(F)$ is contained entirely in some component of $\bigcup_{\j\in \D^n\setminus \{\i\}} \vp_{\j}(F)$.
    This contradicts the essentiality of $\i$ and hence the set $J:=\{j\in\D\setminus\{i_1\}:\vp_j(F)\cap\vp_{i_1i_2}(F)\neq\varnothing\}$ is non-empty.

    Moreover, suppose that $\vp_j(F)\cap \C\neq\varnothing$ for all $j\in J$. Then $\bigcup_{j\in J}\vp_j(F)\subset\C$ and hence $\bigcup_{j\in J\cup\{i_1\}} \vp_j(F)\subset B_1$. Note that
    \[
        \vp_i(F)\cap A \subset \vp_i(F)\cap\vp_{i_1i_2}(F) =\varnothing, \quad \forall i\in\D\setminus (J\cup\{i_1\}).
    \]
    Applying Lemma~\ref{lem:connected} again, we see that
    \[
        H:=\C \cup \Big( \bigcup_{i\in \D\setminus(J\cup\{i_1\})}\vp_i(F) \Big) = (B_1\setminus A) \cup \Big( \bigcup_{i\in\D\setminus(J\cup\{i_1\})}\vp_i(F) \Big)
    \]
    is connected. Since $\bigcup_{j\in J}\vp_j(F)\subset\C$, $H$ contains all level-$1$ cells except $\vp_{i_1}(F)$. Since $H\subset \bigcup_{\j\in \D^n\setminus \{\i\}} \vp_{\j}(F)$, this contradicts the essentiality of $\i$.
\end{proof}


\begin{lemma}\label{lem:pra3}
    Let $i_*\in \D$ be as in Lemma~\ref{lem:pra2}. If there exists $j\in\D\setminus\{i_1,i_*\}$ such that $\vp_{j}(F)\cap\vp_{i_1i_2}(F)\neq\varnothing$, then $\vp_{j}(F)\cap\vp_{i_*}(F)\neq\varnothing$ and $\vp_{j}(F)\cap \C=\varnothing$.
\end{lemma}
\begin{proof}
    If there is such a digit $j$, then both of $\vp_j(F)$ and $\vp_{i_*}(F)$ are level-$1$ cells which intersect the level-$2$ cell $\vp_{i_1i_2}(F)$. Thus $\vp_{i_1i_2}([0,1]^2)$ must locate at one of the corners of the square $\vp_{i_1}([0,1]^2)$, i.e., $i_2\in\{(0,0),(N-1,0),(0,N-1),(N-1,N-1)\}$. From $\vp_j(F)\cap \vp_{i_1i_2}(F)\not=\varnothing$ and Lemma~\ref{lem:fourvertex}, $\vp_{i_1}(\frac{i_2}{N-1})\in \vp_j(F)$. Similarly, $\vp_{i_1}(\frac{i_2}{N-1})\in \vp_{i_*}(F)$. As a result, $\vp_{i_1}(\frac{i_2}{N-1})\in \vp_{j}(F)\cap \vp_{i_*}(F)$, which implies that $\vp_j(F)$ and $\vp_{i_*}(F)$ belong to the same connected component of $\bigcup_{\j\in \D^n\setminus \{\i\}} \vp_{\j}(F)$. Combining this with the fact that $\vp_{i_*}(F)\cap \C=\varnothing$, we have $\vp_{j}(F)\cap \C=\varnothing$.
\end{proof}

For each $i\in \D\setminus \{i_1\}$ satisfying $\vp_i(F)\cap \vp_{i_1i_2}(F)\not=\varnothing$ and $\vp_i(F)\cap \C=\varnothing$, we define
\[
    I(i):=\{ij: j\in\D, \vp_{ij}(F)\cap\vp_{i_1i_2}(F)\neq\varnothing\}.
\]
Now let us prove Proposition~\ref{prop:stillgood} by a case-by-case discussion on the cardinality of $I(i)$. 

\textbf{Case 1}. For each $i\in \D$ satisfying $\vp_i(F)\cap \vp_{i_1i_2}(F)\not=\varnothing$ and $\vp_i(F)\cap \C=\varnothing$, we have $|I(i)|=1$. Fix any such digit and denote it by $i_*$. Then there is exactly one level-$2$ cell in $\vp_{i_*}(F)$ that intersects $\vp_{i_1i_2}(F)$. Recall that $\vp_{i_*}(F)\cap\C=\varnothing$ and $\vp_{i_1}(E)\subset\C$, where $E$ is as in~\eqref{eq:notationofe}. So $\vp_{i_*}(F)\cap\vp_{i_1}(E)=\varnothing$, i.e., $\vp_{i_1i_2}(F)$ is also the only level-2 cell in $\vp_{i_1}(F)$ which intersects $\vp_{i_*}(F)$. By Lemma~\ref{lem:casedisfact}, $\vp_{i_1}(F)\cap\vp_{i_*}(F)$ is a singleton, say $\{x_*\}$. Let
\[
    \D_1 = \{i\in\D: i\D^{n-1} \text{ and } i_*\D^{n-1} \text{ belong to the same connected component of } \Gamma_n-\{\i\}\}
\]
and let $\D_2=\D\setminus\D_1$. Note that $i_1\in\D_2$.

Here is an observation: If $i\in\D\setminus\{i_1\}$ satisfies $\vp_i(F)\cap\vp_{i_1}(F)\neq\varnothing$ but $\vp_i(F)\cap\vp_{i_1i_2}(F)=\varnothing$, then $i\in\D_2
$. In fact, note that for every such $i$,
\[
    \vp_i(F) \cap \vp_{i_1}(E) = \vp_i(F) \cap \vp_{i_1}\Big( \bigcup_{j\in\D\setminus\{i_2\}}\vp_j(F) \Big) = \vp_i(F)\cap\vp_{i_1}(F)\neq\varnothing.
\]
Since $\vp_{i_1}(E)\subset\C$, we have $\vp_i(F)\subset\C$. But recall that $\vp_{i_*}(F)\cap\C=\varnothing$. Thus $i\in\D_2$. As a result,
\begin{align}
    \Big( \bigcup_{j\in\D_1}\vp_{j}(F) \Big) \cap \Big( \bigcup_{j\in\D_2}\vp_{j}(F) \Big) &= \Big( \bigcup_{j\in\D_1}\vp_{j}(F) \Big) \cap \vp_{i_1}(F) \notag \\
    &= \Big( \bigcup_{\substack{j\in\D_1 \\ \vp_{j}(F)\cap\vp_{i_1i_2}(F)\neq\varnothing}}\vp_j(F) \Big)\cap \vp_{i_1}(F). \label{eq:jind1jind2}
\end{align}

\textbf{Subcase 1.1}. The digit $i_*$ is the only element in $\D\setminus\{i_1\}$ such that $\vp_{i_*}(F)\cap\vp_{i_1i_2}(F)\neq\varnothing$. In this case, we immediately have by~\eqref{eq:jind1jind2} that
\[
    \Big( \bigcup_{j\in\D_1}\vp_{j}(F) \Big) \cap \Big( \bigcup_{j\in\D_2}\vp_{j}(F) \Big) = \vp_{i_*}(F) \cap \vp_{i_1}(F) = \{x_*\}.
\]
Thus $F$ is fragile, which leads to a contradiction.

\textbf{Subcase 1.2}. There is some digit $i^*\in\D\setminus\{i_1,i_*\}$ such that $\vp_{i^*}(F)\cap\vp_{i_1i_2}(F)\neq\varnothing$.
We claim that $\vp_{i^*}(F)\cap\vp_{i_1}(F)=\{x_*\}$ and $i^*\in \D_1$ for every such $i^*$. In fact, in this subcase, $\vp_{i_1i_2}([0,1]^2)$ should locate at one of the corners of $\vp_{i_1}([0,1]^2)$. By Lemma~\ref{lem:fourvertex}, we have $\vp_{i_1}(\frac{i_2}{N-1})\in\vp_{i_*}(F)\cap\vp_{i_1}(F)$ and $\vp_{i_1}(\frac{i_2}{N-1})\in\vp_{i^*}(F)\cap\vp_{i_1}(F)$. Thus $x_*=\vp_{i_1}(\frac{i_2}{N-1})$ and $\vp_{i_*}(F)\cap\vp_{i^*}(F)\neq\varnothing$. This further indicates that $\vp_{i^*}(F)\cap\C=\varnothing$, since $\vp_{i_*}(F)\cap\C=\varnothing$ and $\C$ is a connected component. In conclusion, the digit $i^*$ satisfies that $\vp_{i^*}(F)\cap \vp_{i_1i_2}(F)\not=\varnothing$ and $\vp_{i^*}(F)\cap \C=\varnothing$. By our original assumption of Case 1, $|I(i^*)|=1$. Using the same arguments as in the beginning of Case 1, $\vp_{i^*}(F)\cap \vp_{i_1}(F)$ is a singleton so that $\vp_{i^*}(F)\cap \vp_{i_1}(F)=\{x_*\}$. In particular, $\vp_{i^*}(F)\cap\vp_{i_*}(F)\neq\varnothing$, so $i^*\D^{n-1}$ and $i_*\D^{n-1}$ belong to the same connected component of $\Gamma_n-\{\i\}$. In other words, $i^*\in\D_1$, which completes the proof of the claim.

From the claim and \eqref{eq:jind1jind2},
\begin{equation*}
    \Big( \bigcup_{j\in\D_1}\vp_j(F) \Big) \cap \Big( \bigcup_{j\in\D_2}\vp_j(F) \Big) = \Big( \bigcup_{\substack{j\in\D_1 \\ \vp_{j}(F)\cap\vp_{i_1i_2}(F)\neq\varnothing}}\vp_j(F) \Big)\cap \vp_{i_1}(F) = \{x_*\}.
\end{equation*}
So $F$ is fragile and we again arrive at a contradiction.

\textbf{Case 2}. There exists $i_*\in \D$ satisfying the conditions in Lemma~\ref{lem:pra2} with $|I(i_*)|=2$. Rotating and reflecting if necessary, Figures~\ref{fig:pra_case2+3}(A),(B) illustrate all possibilities. Recall the notation $E$ in~\eqref{eq:notationofe}.

\textbf{Subcase 2.1}. Consider the case as in Figure~\ref{fig:pra_case2+3}(A), where $i_2=(b,0)$ and $I=\{i_*(b-1,N-1),i_*(b+1,N-1)\}$ for some $0<b<N-1$. Since $\vp_{i_1i_2}(F)$ intersects both $\vp_{i_*}(\vp_{(b-1,N-1)}(F))$ and $\vp_{i_*}(\vp_{(b+1,N-1)}(F))$,
\[
    \{(0,0),(N-1,0),(0,N-1),(N-1,N-1)\} \subset\D.
\]
Note that at least one of $(0,0)$ and $(N-1,0)$ is not $i_2$, say $(0,0)$. Then
\[
    \C\cap\vp_{i_*}(F) \supset \vp_{i_1}(E)\cap\vp_{i_*}(F) \supset \vp_{i_1}(\vp_{(0,0)}(F)) \cap \vp_{i_*}(\vp_{(0,N-1)}(F))\neq\varnothing.
\]
This is a contradiction.

\begin{figure}[htbp]
    \centering
    \subfloat[Subcase 2.1]
    {
        \begin{minipage}[t]{130pt}
            \centering
            \begin{tikzpicture}[scale=1.3]
                \draw[thick] (0,0) rectangle (2.5,2.5);
                \draw[thick] (0,-2.5) rectangle (2.5,0);
                \draw[thick] (1,0) rectangle (1.5,0.5);
                \draw[thick] (1,0) rectangle (1.1,0.1);
                \draw[thick] (0.5,-0.5) rectangle (1,0);
                \draw[thick] (1.5,-0.5) rectangle (2,0);
                \draw[thick] (0.9,-0.1) rectangle (1,0);
                \draw[thick] (1.5,-0.1) rectangle (1.6,0);
                \draw[thick] (1.4,0) rectangle (1.5,0.1);
                \node[font=\fontsize{20}{1}\selectfont] at(1.25,1.25) {$i_1$};
                \node[font=\fontsize{20}{1}\selectfont] at(1.25,-1.25) {$i_*$};
                \node[font=\fontsize{10}{1}\selectfont] at(1.25,0.3) {$i_1i_2$};
            \end{tikzpicture}
        \end{minipage}
    }
    \subfloat[Subcase 2.2]
    {
        \begin{minipage}[t]{130pt}
            \centering
            \begin{tikzpicture}[scale=1.3]
                \draw[pattern=north west lines] (0,0)--(0.5,0)--(0.5,1)--(1.5,1)--(1.5,0.5)--(2.5,0.5)--(2.5,2.5)--(1,2.5)--(1,2)--(0,2)--(0,0);
                \draw[thick] (1.8,0) rectangle (1.9,0.1);
                \draw[thick] (1.7,-0.1) rectangle (1.8,0);
                \draw[thick] (1.7,0.4) rectangle (1.8,0.5);
                \draw[thick] (1.8,0.4) rectangle (1.9,0.5);
                \draw[thick] (0,0) rectangle (2.5,2.5);
                \draw[thick] (0,-2.5) rectangle (2.5,0);
                \draw[thick] (1.5,0) rectangle (2,0.5);
                \draw[thick] (1.5,0) rectangle (1.6,0.1);
                \draw[thick] (1,-0.5) rectangle (1.5,0);
                \draw[thick] (1.5,-0.5) rectangle (2,0);
                \draw[thick] (1.4,-0.1) rectangle (1.5,0);
                \draw[thick] (1.8,-0.1) rectangle (1.9,0);
                \node[font=\fontsize{20}{1}\selectfont] at(1.25,1.25) {$i_1$};
                \node[font=\fontsize{20}{1}\selectfont] at(1.25,-1.25) {$i_*$};
                \node[font=\fontsize{10}{1}\selectfont] at(1.75,0.25) {$i_1i_2$};
            \end{tikzpicture}
        \end{minipage}
    }
    \subfloat[Case 3]
    {
        \begin{minipage}[t]{130pt}
            \centering
            \begin{tikzpicture}[scale=1.3]
                \draw[thick] (1.5,0) rectangle (4,2.5);
                \draw[thick] (1.5,-2.5) rectangle (4,0);
                \draw[thick] (2.5,0) rectangle (3,0.5);
                \draw[thick] (2.5,0) rectangle (2.6,0.1);
                \draw[thick] (2,-0.5) rectangle (2.5,0);
                \draw[thick] (3,-0.5) rectangle (3.5,0);
                \draw[thick] (2.4,-0.1) rectangle (2.5,0);
                \draw[thick] (3,-0.1) rectangle (3.1,0);
                \draw[thick] (2.9,0) rectangle (3,0.1);
                \draw[thick] (2.5,-0.5) to (3,-0.5);
                \node[font=\fontsize{20}{1}\selectfont] at(2.75,1.25) {$i_1$};
                \node[font=\fontsize{20}{1}\selectfont] at(2.75,-1.25) {$i_*$};
                \node[font=\fontsize{10}{1}\selectfont] at(2.75,0.3) {$i_1i_2$};
            \end{tikzpicture}
        \end{minipage}
    }
    \caption{ Subcases 2.1, 2.2 and Case 3, where the shaded region in (B) illustrates $\vp_{i_1}(E)$}
    \label{fig:pra_case2+3}
\end{figure}
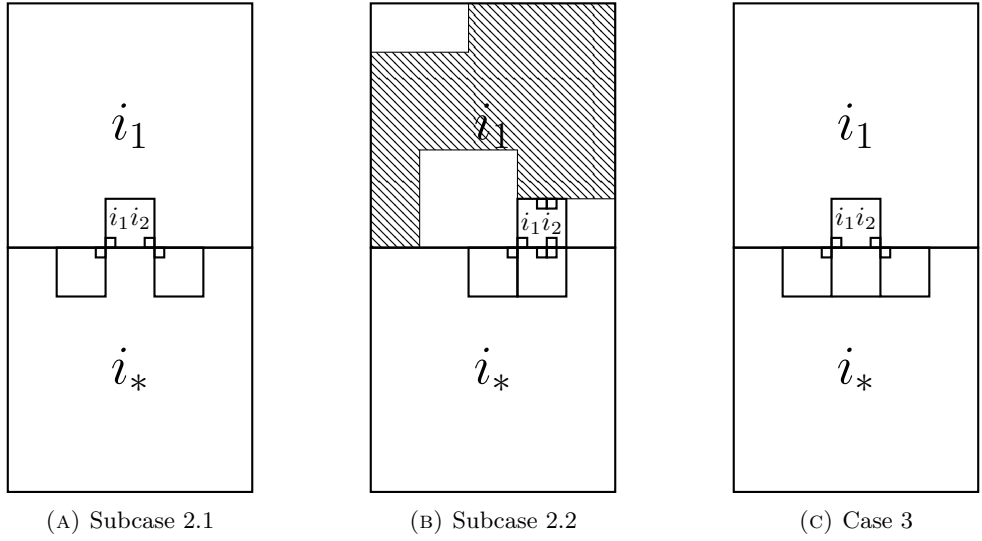

\textbf{Subcase 2.2}. Consider the case as in Figure~\ref{fig:pra_case2+3}(B), where $i_2=(a,0)$ and $I=\{i_*(a-1,N-1),i_*(a,N-1)\}$ for some $0<a\leq N-1$. In particular, $(a-1,N-1),(a,N-1)\in\D$. Moreover, since $\vp_{i_*}(\vp_{(a-1,N-1)}(F))\cap\vp_{i_1i_2}(F)\neq\varnothing$, we have $(0,0),(N-1,N-1)\in\D$. Since $\vp_{i_1}(E)\cap\vp_{i_*}(F)=\varnothing$, we see that $(a-1,0),(a+1,0)\notin\D$ and
\begin{equation}\label{eq:updown}
   \vp_{i_1}(F) \cap \vp_{i_*}(F) = \vp_{i_1}(\vp_{(a,0)}(F))\cap ( \vp_{i_*}(\vp_{(a-1,N-1)}(F)\cup\vp_{(a,N-1)}(F)) ).
\end{equation}

We will prove by contradiction that $i_3\cdots i_n$ is an  essential cut vertex of $\Gamma_{n-2}$ in this subcase. Otherwise, $\bigcup_{j\in\D\setminus\{i_3\}} \vp_j(F)$ is contained in exactly one connected component of $\bigcup_{\j\in\D^{n-2}\setminus\{i_3\cdots i_n\}} \vp_{\j}(F)$. Thus we can find a component $\C'$ of $\bigcup_{\j\in\D^n\setminus\{\i\}}\vp_{\j}(F)$ which contains $\vp_{i_1i_2}(\bigcup_{j\in\D\setminus\{i_3\}} \vp_j(F))$. Note that
\begin{align*}
    \C'\cap\vp_{i_*}(F) &\supset \vp_{i_1i_2}\Big(\bigcup_{j\in\D\setminus\{i_3\}} \vp_j(F)\Big)\cap\vp_{i_*}(F) \\
    &\supset \vp_{i_2i_2}\Big( \bigcup_{j\in\{(0,0),(a,0)\}\setminus\{i_3\}}\vp_j(F) \Big)\cap\vp_{i_*}(F) \neq \varnothing.
\end{align*}
Recall that $\vp_{i_*}(F)\cap\C=\varnothing$. Thus $\C'\cap\C=\varnothing$, i.e., they are different components of $\bigcup_{\j\in\D^n\setminus\{\i\}}\vp_{\j}(F)$.

{\bf Claim}. $(a,1)\notin\D$, $(a+1,1)\in\D$ and $\vp_{(a-1,1)}(F)\cap\vp_{(a,0)}(F)=\varnothing$ (if $(a-1,1)\in\D$).

Firstly, if $(a,1)\in\D$ then $(a,1)\in\D\setminus\{i_2\}$, and we have by~\eqref{eq:updown} and the self-similarity of $F$ that
\begin{align*}
    \C'\cap\C &\supset \vp_{i_1i_2}\Big(\bigcup_{i\in\D\setminus\{i_3\}}\vp_i(F)\Big) \cap \vp_{i_1}\Big(\bigcup_{j\in\D\setminus\{i_2\}}\vp_{j}(F)\Big) \\
    &\supset \vp_{i_1i_2}\Big(\bigcup_{i\in\{(a-1,N-1),(a,N-1)\}\setminus\{i_3\}}\vp_i(F)\Big) \cap \vp_{i_1}(\vp_{(a,1)}(F)) \neq\varnothing,
\end{align*}
which leads to a contradiction. Secondly, if $\vp_{(a-1,1)}(F)\cap\vp_{(a,0)}(F)\neq\varnothing$ then $(N-1,0),(0,N-1)\in\D$, and we will obtain a contradiction as in Subcase 2.1. Finally, combining these two observations and the fact that $(a-1,0),(a+1,0)\notin\D$, we must have $(a+1,1)\in\D$ (otherwise, $F$ is disconnected). This establishes the claim.

But now it is easy to see that
\[
    \vp_{(a,0)}(F) \cap \Big( \bigcup_{i\in\D\setminus\{(a,0)\}}\vp_i(F) \Big) = \vp_{(a,0)}(F) \cap \vp_{(a+1,1)}(F)
\]
is a singleton. This means that $F$ is fragile and we again arrive at a contradiction.

\textbf{Case 3}. There exists $i_*\in \D$ satisfying the conditions in Lemma~\ref{lem:pra2} with $|I(i_*)|=3$. Rotating and reflecting if necessary, Figures~\ref{fig:pra_case2+3}(C) illustrates all possibilities. In this case, we again have
\[
    \{(0,0),(N-1,0),(0,N-1),(N-1,N-1)\} \subset\D.
\]
and will obtain a contradiction as in Subcase 2.1.

\section{Possible numbers of cut points}
\subsection{Constructions}
It is also interesting to consider the possible number of cut points of any connected GSC $F$. We have the following possibilities.
\begin{enumerate}
    \item $F$ has no cut points (e.g., the standard Sierpi\'nski carpet).
    \item $F$ has exactly one cut point. For example, the GSC in Figure~\ref{fig:exa_goodcp} satisfies this requirement, i.e., $(1/2,0)$ is the only cut point of $F$. Please see Lemma~4.1 in \cite{RW17} for a detailed proof.
    \item $F$ has more than one but still finitely many cut points. Please see Example~\ref{exa:exactlyncutpts}.
    \item $F$ has countably many cut points. For example, take $N=3$ and \[ \D=\{(0,1),(1,0),(1,1),(1,2),(2,0),(2,1),(2,2)\}. \] Please see Figure~\ref{fig:numexa}. In this case, the collection of cut points of $F$ is a subset of
    \[
        \{(1/2,1/3)\} \cup \{\vp_{\i}((1/2,1/3)): \i\in\D^*\},
    \]
    and it is easy to see that $\{\vp_{(1,0)}^n((1/2,1/3)): n\geq 1\}$ are cut points of $F$. One can modify the argument in Example~\ref{exa:exactlyncutpts} (by considering the union of all horizontal and vertical line segments instead) and get a rigorous proof.
    \item $F$ has uncountably many cut points. This happens trivially when $F$ is a line segment. For example, take $N=3$ and $\D=\{(0,0),(1,0),(2,0)\}$ (so $F=F(N,\D)$ is just the interval $[0,1]\times\{0\}$).
\end{enumerate}

\begin{figure}[htbp]
        \centering
        \includegraphics[width=3.7cm]{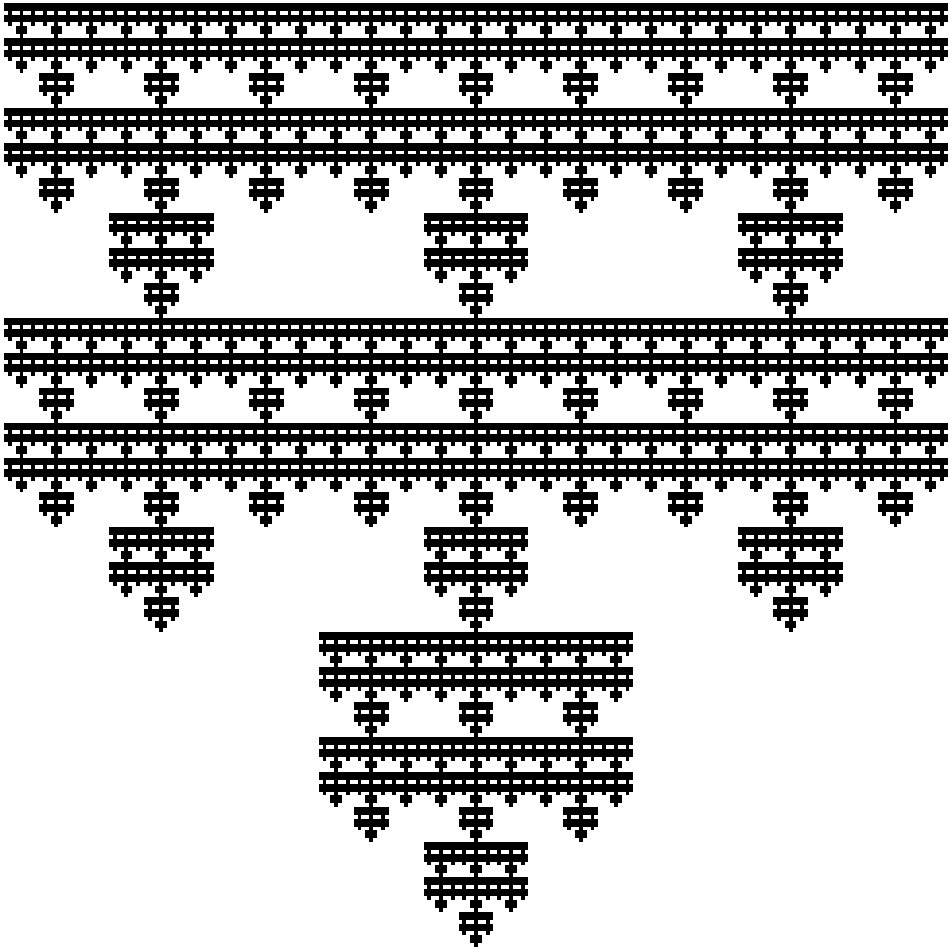}
    \caption{A GSC with countably many cut points}
    \label{fig:numexa}
\end{figure}

The following example establishes Theorem~\ref{thm:number}.


\begin{example}\label{exa:exactlyncutpts}
    Given any positive integer $m\geq 3$, we can construct a connected GSC with exactly $(2m-3)$ cut points. For example, let
    \[
        \Lambda_k = \{(k,i):0\leq i\leq m-1\} \cup \{(k+1,i): m\leq i\leq 2m-1\}, \quad 0\leq k\leq 2m-2,
    \]
    and set
    \[
        \D_m = \{(0,i):m\leq i\leq 2m-1\} \cup \{(2m-1,i):0\leq i\leq m-1\} \cup \bigcup_{k=0}^{m-1} \Lambda_{2k}.
    \]
    We claim that the GSC $F=F(2m,\D_m)$ has exactly $(2m-3)$ cut points. Please see Figure~\ref{fig:3cps}(A) for the case when $m=3$.

    Here is a sketch of proof of the above claim, which is a modification of the proof of~\cite[Lemma 4.1]{RW17}. Note that there are infinitely many line segments of slope $m$ in $F$. Denote $A$ to be the union of all line segments in $F$ of slope $m$ and $\infty$ (i.e., vertical ones). It is not hard to see that $A$ is a connected subset of $F$, and $A\setminus\{x\}$ remains connected unless
    \[
        x \in \Big\{ \Big( \frac{i}{2m}, \frac{1}{2} \Big): 2\leq i\leq 2m-2 \Big\} =: C.
    \]
    For any $x\notin C$ and any $y\in F\setminus\{x\}$, there is some $\i\in\D_m^*$ such that $y\in\vp_{\i}(F)$ but $x\notin\vp_{\i}(F)$. Since $\vp_{\i}(F)$ contains infinitely many line segments of slope $m$ and $\infty$, $\vp_{\i}(F)$ and $A\setminus\{x\}$ (which are both connected) belong to the same connected component of $F\setminus\{x\}$. In particular, $y$ and $A\setminus\{x\}$ belong to the same connected component of $F\setminus\{x\}$. It then follows from the arbitrariness of the choice of $y$ that $F\setminus\{x\}$ is connected. On the other hand, it is easy to see that every point in $C$ is indeed a cut point of $F$. Thus $F$ contains exactly $|C|=2m-3$ cut points.

    The above construction settles the existence of GSCs containing an odd number ($\geq 3$) of cut points. For even numbers, just set
    \[
        \D'_m = \D_m\cup\{(2m-2,i): m\leq i\leq 2m-1\}.
    \]
    Similarly as above, it is not hard to see that the new GSC $F'=F'(2m,\D')$ contains exactly $(2m-4)$ cut points
    \[
        \Big( \frac{2}{2m}, \frac{1}{2} \Big), \Big( \frac{3}{2m}, \frac{1}{2} \Big), \ldots, \Big( \frac{2m-3}{2m}, \frac{1}{2} \Big).
    \]
    Please see Figure~\ref{fig:3cps}(B) for the case when $m=3$.
    \begin{figure}[htbp]
        \centering
        \subfloat[$3$ cut points]
        {
            \begin{minipage}[t]{130pt}
                \centering
                \includegraphics[width=4cm]{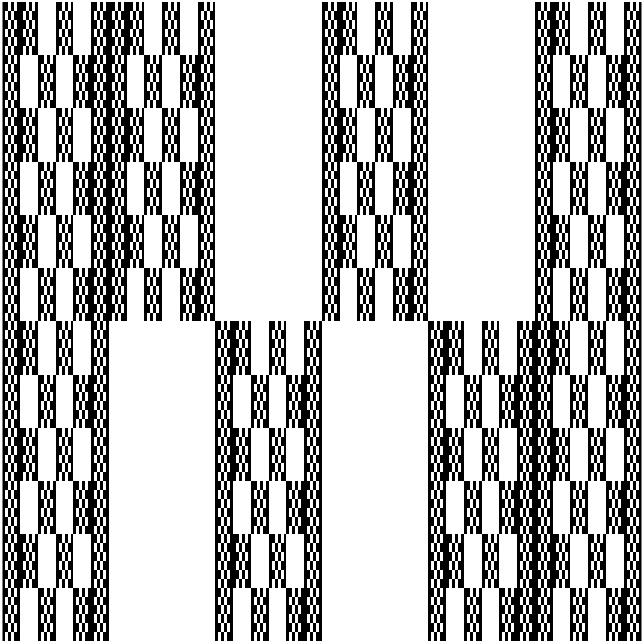}
            \end{minipage}
        }
        \subfloat[$2$ cut points]
        {
            \begin{minipage}[t]{130pt}
                \centering
                \includegraphics[width=4cm]{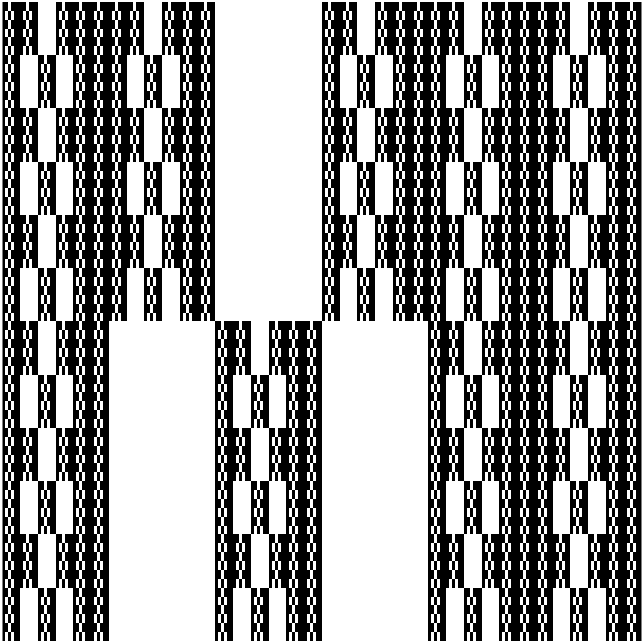}
            \end{minipage}
        }
        \caption{GSCs with exactly $3$ and $2$ cut points}
        \label{fig:3cps}
    \end{figure}
\end{example}

\subsection{Non-fragile cases}
In this subsection, we prove Theorem~\ref{thm:2orinfinity}.
To this end, we need a series of observations as follows.

\begin{lemma}\label{lem:finitecomandnont}
    For any $x\in F$, $F\setminus\{x\}$ has at most eight connected components. Moreover, none of these components is trivial.
\end{lemma}
\begin{proof}
    Let $\bi=i_1i_2\cdots\in\D^\infty$ be an infinite word such that $\{x\}=\bigcap_{n=1}^\infty \vp_{\bi|_n}(F)$. Writing $E_n=\bigcup_{\j\in\D^n\setminus\{\bi|_n\}}\vp_{\j}(F)$, it is easy to see that $\{E_n\}$ is an increasing sequence and $F\setminus\{x\}=\bigcup_{n=1}^\infty E_n$. Suppose on the contrary that we can find nine distinct connected components $\C_1,\ldots,\C_9$ of $F\setminus\{x\}$. Select $z_i\in\C_i$. Note that $\{z_i\}_{i=1}^9\subset E_{n_0}$ for some large $n_0$. Since $E_{n_0}\subset F\setminus\{x\}$, $z_i$ and $z_j$ must belong to different components of $E_{n_0}$ whenever $i\neq j$. So $E_{n_0}$ contains at least nine components. Equivalently (by Lemma~\ref{lem:equivalent}), $\Gamma_{n_0}-\{\bi|_{n_0}\}$ has at least nine components. Since $\bi|_{n_0}$ has at most eight neighbors in $\Gamma_{n_0}$, this contradicts Lemma~\ref{lem:graphcon}.

    For the second statement, suppose that $\{z\}$ is a trivial connected component of $F\setminus\{x\}$. Since $z\neq x$, we can find a word $\i\in\D^n$ for some large $n$ such that $z\in\vp_{\i}(F)$ but $x\notin\vp_{\i}(F)$. The connectedness of $F$ implies that $\vp_{\i}(F)$ is also connected. So $z$ and $\vp_{\i}(F)$ belong to the same connected component of $F\setminus\{x\}$. This is a contradiction.
\end{proof}

\begin{lemma}\label{lem:pathcomponent}
    For any $x\in F$, the space $F\setminus\{x\}$ is locally path connected. In particular, the connected components and the path connected components of $F\setminus\{x\}$ are the same.
\end{lemma}
\begin{proof}
    We claim that for every open set $U$ of $F\setminus\{x\}$, each path connected component of $U$ is open in $F\setminus\{x\}$. As a result, for any $y\in U$, the path connected component of $U$ containing $y$ is a neighborhood of $y$. Since $U$ is arbitrary, $F\setminus\{x\}$ is locally path connected.

    For the claim, let $\C$ be a path component of $U$ and let $z\in \C$. Since $U$ is open and $z\neq x$, we can choose $n$ so large that for every $\i\in\D^n$ with $z\in\vp_{\i}(F)$, $\vp_{\i}(F)\subset U \subset F\setminus\{x\}$. Recall that a connected self-similar set is always path connected (please refer to~\cite[Theorem 1.6.2]{Kig01}). So for every such $\i$, $\vp_{\i}(F)$ is path connected and hence $\vp_{\i}(F)\subset\C$.

    Moreover, since $\bigcup_{\j\in\D^n: z\notin\vp_{\j}(F)}\vp_{\j}(F)$ is a compact subset of $\R^2$, the distance between it and the singleton $\{z\}$ is a positive real number, say $\delta$. Then for $0<r<\delta$,
    \[
        \{y\in F\setminus\{x\}: |y-z|<r\} \subset \bigcup_{\j\in\D^n:z\in\vp_{\j}(F)} \vp_{\j}(F)\subset\C.
    \]
    This indicates the openness of $\C$.
\end{proof}

Motivated by~\cite{Xiao21}, we introduce the following definition.

\begin{definition}
    Let $F$ be a connected GSC. Suppose that $x\in F$ and $\C$ is a connected component of $F\setminus\{x\}$.
    \begin{enumerate}
        \item The component $\C$ is called \emph{vertical} if $\C$ meets both $[0,1]\times\{0\}$ and $[0,1]\times\{1\}$;
        \item The component $\C$ is called \emph{horizontal} if $\C$ meets both $\{0\}\times[0,1]$ and $\{1\}\times[0,1]$;
        \item The component $\C$ is called \emph{corner-like} if it is neither vertical nor horizontal.
    \end{enumerate}
\end{definition}

We also call any path $\ell\subset[0,1]^2$ \emph{vertical} (resp. \emph{horizontal}) if $\ell$ meets both $[0,1]\times\{0\}$ and $[0,1]\times\{1\}$ (resp. $\{0\}\times[0,1]$ and $\{1\}\times[0,1]$). As a consequence of Lemma~\ref{lem:pathcomponent}, every vertical (resp. horizontal) component of $F\setminus\{x\}$, where $x\in F$, contains a vertical (resp. horizontal) path. Recall that $C_F$ denotes the set of cut points of $F$.

\begin{proposition}\label{prop:cornermeansinf}
    If there is some $x\in C_F$ and some connected component of $F\setminus\{x\}$ that is corner-like, then $\#C_F=+\infty$.
\end{proposition}
\begin{proof}
    Denote the corner-like connected component in the assumption by $\C$. So $\C$ is neither vertical nor horizontal. Without loss of generality, assume that $\C\cap([0,1]\times\{1\})=\varnothing$ and $\C\cap(\{1\}\times[0,1])=\varnothing$. Let $a_*=\min\{a: (a,b)\in\D\}$, $b_*=\min\{b:(a_*,b)\in\D\}$ and write $i_*=(a_*,b_*)\in\D$. Recalling Definition~\ref{de:position}, we have $i_*(\nwarrow),i_*(\leftarrow), i_*(\swarrow), i_*(\downarrow)\notin\D$. Writing $y=\vp_{i_*}(x)$, it is easy to see that
    \[
       \vp_{i_*}(\C)\cap \Big( (\vp_{i_*}(F)\setminus \{y\}) \setminus \{\vp_{i_*}(\C)\} \Big)=\varnothing.
    \]
    Thus 
    \begin{align*}
        \vp_{i_*}(\C) \cap \Big( (F\setminus\{y\}) \setminus\vp_{i_*}(\C) \Big) &\subset \vp_{i_*}(\C) \cap \Big( \bigcup_{i\in\D\setminus\{i_*\}}\vp_i(F) \Big) \\
        &\subset \vp_{i_*}(\C) \cap \Big( \bigcup_{t\in\P}\vp_{i_*(t)}(F) \Big) \\
        &\subset \vp_{i_*}(\C) \cap \Big( \bigcup_{t\in\{\uparrow,\nearrow,\rightarrow,\searrow\}}\vp_{i_*(t)}(F) \Big) \\
        &= \vp_{i_*}(\C) \cap (\vp_{i_*}([0,1]\times\{1\}) \cup\vp_{i_*}(\{1\}\times[0,1]) ) = \varnothing,
    \end{align*}
    implying that $\vp_{i_*}(\C)$ is a connected component of $F\setminus\{y\}$. In particular, $y\in C_F$. Since $\C$ is nontrivial (recall Lemma~\ref{lem:finitecomandnont}) and $\vp_{i_*}$ is strictly contractive, $\vp_{i_*}(\C)\neq\C$.

    Further, it follows from the self-similarity of $F$ that $\vp_{i_*}^2(\C)$ is a connected component of $\vp_{i_*}(F\setminus\{y\})=\vp_{i_*}(F)\setminus\{\vp_{i_*}(y)\}$. The proof of Lemma~\ref{lem:iiawayfromj} actually shows that
    \[
        \dist(\vp_{i_*}^2(\C), \vp_{i}(F)) \geq \dist(\vp_{i_*}^2(F), \vp_{i}(F)) >0, \quad \forall i\in\D\setminus\{i_*\}.
    \]
    Since $F=\vp_{i_*}(F)\cup\bigcup_{i\in\D\setminus\{i_*\}}\vp_i(F)$, $\vp_{i_*}^2(\C)$ should be a connected component of $F\setminus\{\vp_{i_*}(y)\}$, and it is not equal to $\C$ and $\vp_{i_*}(\C)$ for the reason as before. By an easy induction process, we can show that $\vp_{i_*}^n(\C)$ is a connected component of $F\setminus\{\vp_{i_*}^{n-1}(y)\}$ for all $n\geq 1$, and $\vp_{i_*}^n(\C)\neq\vp_{i_*}^m(\C)$ whenever $n\neq m$. Combining with Lemma~\ref{lem:finitecomandnont}, $\#C_F$ must be infinity.
\end{proof}

\begin{corollary}\label{cor:allverorhor}
    If $0<\#C_F<\infty$, then for every $x\in C_F$, the connected components of $F\setminus\{x\}$ are either all vertical or all horizontal.
\end{corollary}
\begin{proof}
    Since $0<\#C_F<\infty$, the above proposition tells us that every connected component of $F\setminus\{x\}$ is either vertical or horizontal. Suppose there are two components $\C,\C'$ of $F\setminus\{x\}$ such that $\C$ is vertical but $\C'$ is horizontal. By Lemma~\ref{lem:pathcomponent}, $\C$ and $\C'$ are both path connected. So there is a vertical path in $\C$ and a horizontal path in $\C'$. But these two paths must intersect with each other, implying that $\C\cap\C'\neq\varnothing$ and hence $\C=\C'$. This is a contradiction.
\end{proof}

As a result, for any $x\in C_F$ and any pair of vertical connected components $\C,\C'$ of $F\setminus\{x\}$, $\C$ is either to the left or to the right of $\C'$. More precisely, we have either
\[
    \sup\{x: (x,y)\in\C\} \leq \inf\{x: (x,y)\in\C'\} \quad \text{ for all } y\in\{0,1\}
\]
or vice versa.

\begin{proposition}\label{prop:twovertical}
    If $0<\# C_F<\infty$, then for every $x\in C_F$, $F\setminus\{x\}$ contains exactly two connected components that are both vertical or both horizontal.
\end{proposition}
\begin{proof}
    Let $\bi=i_1i_2\cdots\in\D^\infty$ be an infinite word such that $\{x\}=\bigcap_{n=1}^\infty \vp_{\i|_n}(F)$, and recall that $F\setminus\{x\}=\bigcup_{n=1}^\infty E_n$, where $E_n:=\bigcup_{\j\in\D^n\setminus\{\bi|_n\}}\vp_{\j}(F)$ is as in the proof of Lemma~\ref{lem:finitecomandnont}. By the above corollary, we may assume without loss of generality that all connected components of $F\setminus\{x\}$ are vertical.

    Let $\C$ be a component of $F\setminus\{x\}$. Since $\C$ is vertical, it cannot be contained in one level-$1$ cell. In particular, $\C\nsubseteq\vp_{i_1}(F)$. So for any $n\geq 1$, $\C\cap E_n\not=\varnothing$. Combining this with $E_n\subset F\setminus \{x\}$,  $\C$ must contain some connected component, say $\C_n$, of $E_n$. By definition, there is a connected component of the graph $\Gamma_n-\{\bi|_n\}$ of which the vertex set $V_n$ is such that $\C_n=\bigcup_{\j\in V_n} \vp_{\j}(F)$. Recalling Lemma~\ref{lem:graphcon}, we can further find $\omega_n\in V_n$ that is adjacent to $\bi|_n$, i.e., $\vp_{\omega_n}(F)\cap\vp_{\bi|_n}(F)\neq\varnothing$. Hence
    \[
        \dist(x,\C) \leq \dist(x,\C_n) \leq \dist(x,\vp_{\omega_n}(F)) \leq 2\sqrt{2}N^{-n}, \quad \forall n\geq 1.
    \]
    Combining with the fact that $\C$ is vertical and path connected, we can find vertical paths $\{\ell_n\}_{n=1}^\infty$ in $\C$ such that $\dist(\ell_n,x)\leq 2N^{-n}$.

    Suppose the lemma is false, i.e., there are two vertical (path) connected components $\C', \C''$ of $F\setminus\{x\}$ other than $\C$. Without loss of generality, assume that $\C,\C'',\C'$ lie from left to right. Similarly as above, we can find vertical paths $\{\ell'_n\}_{n=1}^\infty$ in $\C'$ such that $\dist(\ell'_n,x)\leq 2N^{-n}$. For every $n$, select $z_n\in \ell_n$ and $z'_n\in\ell'_n$ with $|x-z_n|\leq 2N^{-n}$ and $|x-z'_n|\leq 2N^{-n}$, respectively. Denoting by $\ell(z_n,z'_n)$ the line segment joining $z_n$ and $z'_n$, it is easy to see that $|z-x|\leq 2N^{-n}$ for all $z\in\ell(z_n,z'_n)$.

    Since $\C''$ is vertical, there is a vertical path $\ell\subset\C''$. Clearly, $\ell\cap\ell_n=\varnothing$ and $\ell\cap\ell'_n=\varnothing$ for all $n$. Since $\C''$ lies in the ``middle'' of $\C$ and $\C'$, $\ell$ must meet $\ell(z_n,z'_n)$ and hence $\dist(\ell,x)\leq 2N^{-n}$ for all $n$. This further tells us that $x\in\ell$, which leads to a contradiction since $\ell\subset\C''\subset F\setminus\{x\}$.
\end{proof}

\begin{corollary}\label{cor:otheralsover}
    Assume that $0<\#C_F<\infty$ and let $x\in C_F$. If connected components of $F\setminus\{x\}$ are all vertical (resp. horizontal), then for any $y\in C_F$, all components of $F\setminus\{y\}$ are also vertical (resp. horizontal).
\end{corollary}
\begin{proof}
    By Lemmas~\ref{lem:pathcomponent} and~\ref{prop:twovertical}, there are exactly two vertical path connected components of $F\setminus\{x\}$. In particular, we can find two disjoint vertical paths $\ell, \ell'\subset F\setminus\{x\}$. Note that at least one of them does not pass $y$. Without loss of generality, assume that $y\notin\ell$. By Corollary~\ref{cor:allverorhor}, connected components of $F\setminus\{y\}$ are all vertical or all horizontal. In the former case, there is nothing to prove. In the latter case, each of those components contains a horizontal path. But all these paths must meet the vertical path $\ell$. Therefore, they belong to the same connected component of $F\setminus\{y\}$, which is a contradiction.
\end{proof}

\begin{lemma}\label{lem:furthercorner}
    Suppose $0<\#C_F<\infty$, $x\in C_F$ and $\C$ is a vertical connected component of $F\setminus\{x\}$. Then
    \[
        \C \cap ((F+(0,1))\setminus\{x\}) \neq\varnothing \quad\text{ and }\quad \C \cap ((F+(0,-1))\setminus\{x\}) \neq\varnothing.
    \]
    In particular, $(F\setminus\{x\})\cup(F+(0,1))$ and $(F\setminus\{x\})\cup(F+(0,-1))$ are both connected.
\end{lemma}
\begin{proof}
    Without loss of generality, suppose on the contrary that $\C$ is vertical but $\C \cap ((F+(0,1))\setminus\{x\}=\varnothing$. By Proposition~\ref{prop:twovertical}, we may also assume that $\C$ is the leftmost component of $F\setminus\{x\}$. Thus
    \[
        \C \cap (F+i)=\varnothing, \quad \forall i\in\{(1,1),(1,0),(1,-1)\}.
    \]
    Letting $i_*=(a_*,b_*)$ be as in the proof of Proposition~\ref{prop:cornermeansinf}, we see that
    \begin{align*}
        \vp_{i_*}(\C) \cap \Big( \bigcup_{i\in\D\setminus\{i_*\}} \vp_i(F)\setminus\{\vp_{i_*}(x)\} \Big) &= \vp_{i_*}(\C) \cap \Big( \bigcup_{t\in\P} \vp_{i_*(t)}(F)\setminus\{\vp_{i_*}(x)\} \Big) \\
        &\subset \vp_{i_*}(\C) \cap (\vp_{i_*(\uparrow)}(F)\setminus\{\vp_{i_*}(x)\}) \\
        &= \vp_{i_*}(\C\cap ((F+(0,1))\setminus\{x\})) = \varnothing.
    \end{align*}
    This implies that $\vp_{i_*}(\C)$ is a connected component of $F\setminus\{\vp_{i_*}(x)\}$. Further, since $\vp_{i_*}(\C)\subset\vp_{i_*}(F)$, it is corner-like. By Proposition~\ref{prop:cornermeansinf}, we obtain a contradiction.

    Combining with the connectedness of $F$ and Proposition~\ref{prop:twovertical}, it is easy to see that the second statement holds.
\end{proof}

\begin{lemma}\label{lem:leftrightside}
    Let $x\in C_F$ and suppose that $F\setminus\{x\}$ contains exactly two vertical connected components $\C_{x,1}$ and $\C_{x,2}$. If $\C_{x,1}$ lies to the left of $\C_{x,2}$, then
    \[
        F\cap(\{0\}\times[0,1]) = \C_{x,1}\cap(\{0\}\times[0,1]),\, F\cap(\{1\}\times[0,1]) = \C_{x,2}\cap(\{1\}\times[0,1]).
    \]
\end{lemma}
\begin{proof}
    We first claim that $x\notin\{0,1\}\times[0,1]$. Suppose on the contrary that $x\in\{0\}\times[0,1]$ and let $\ell_{x,1}\subset\C_{x,1}$ be a vertical path. Then $\delta:=\dist(x,\ell_{x,1})>0$. Since $x$ lies in the leftmost side of the unit square and $\C_{x,2}$ is to the right of $\C_{x,1}$, $\dist(x,\C_{x,2})\geq\dist(x,\ell_{x,1})=\delta>0$. But as in the proof of Proposition~\ref{prop:twovertical}, there are vertical paths in $\C_{x,2}$ that is arbitrarily close to $x$. This is a contradiction. Similarly, $x\notin\{1\}\times[0,1]$.

    Suppose there is some $z\in F\cap(\{0\}\times[0,1])$ but $z\notin\C_{x,1}$. Then $z\in \C_{x,2}$ and we can find a vertical path $\ell_z\subset\C_{x,2}$ passing $z$. Since $\C_{x,1}$ lies to the left of $\C_{x,2}$ and $z$ lies in the leftmost side of the unit square, every vertical path in $\C_{x,1}$ must intersect $\ell_z$. This is impossible since $\C_{x,1}\cap\C_{x,2}=\varnothing$. Now we obtain a contradiction, so $F\cap(\{0\}\times[0,1]) \subset \C_{x,1}$. Similarly, $F\cap(\{1\}\times[0,1]) \subset \C_{x,2}$. This establishes the equalities.
\end{proof}

\begin{proposition}\label{prop:vpixisthenvpiyis}
    Assume that $0<\#C_F<\infty$ and let $j\in\D$ and $x\in C_F$. If $\vp_j(x)\in C_F$ while $\vp_j(x)\notin\vp_i(F)$ for every $i\in\D\setminus\{j\}$, then $\vp_j(y)\in C_F$ for all $y\in C_F$.
\end{proposition}
\begin{proof}
    If $y=x$ then there is nothing to prove, so it suffices to consider when $y\neq x$. By Proposition~\ref{prop:twovertical}, we may assume that $F\setminus\{x\}$ contains exactly two vertical connected components $\C_{x,1}$ and $\C_{x,2}$. Without loss of generality, assume that $\C_{x,1}$ lies to the left of $\C_{x,2}$, and write $x_j=\vp_j(x)$ for convenience.

    We claim that $j\pm(0,1)\notin\D$.
    Suppose on the contrary that $j+(0,1)\in\D$. Then $x_j\notin\vp_{j+(0,1)}(F)$. Note that
    \begin{align*}
        \vp_{j+(0,1)}(F) \cup (\vp_{j}(F)\setminus\{x_j\}) &= \vp_{j+(0,1)}(F) \cup \vp_{j}(F\setminus\{x\}) \\
        &= \frac{(F+(0,1))\cup (F\setminus\{x\})}{N} + \frac{j}{N}
    \end{align*}
    is a scaled copy of $(F+(0,1))\cup (F\setminus\{x\})$. Then, since connected components of $F\setminus\{x\}$ are both vertical, we see from Lemma~\ref{lem:furthercorner} that $\vp_{j+(0,1)}(F) \cup (\vp_{j}(F)\setminus\{x_j\})$ is connected. Similarly, if $j-(0,1)\in\D$ then $\vp_{j-(0,1)}(F) \cup (\vp_{j}(F)\setminus\{x_j\})$ is connected. Writing $\Lambda=\{j,j\pm(0,1)\}\cap\D$, we see that
    \[
        \bigcup_{t\in\Lambda} \vp_t(F)\setminus\{x_j\} = (\vp_j(F)\setminus\{x_j\}) \cup \Big( \bigcup_{t\in\D\cap\{j\pm(0,1)\}}\vp_t(F) \Big)
    \]
    is a connected set. Since
    \[
        F = \Big( \bigcup_{t\in\Lambda} \vp_t(F) \Big) \cup \Big( \bigcup_{t\notin\Lambda} \vp_t(F) \Big) 
    \]
    is connected, applying Lemma~\ref{lem:connected} (to $B_1=\bigcup_{t\in\Lambda}\vp_t(F)$, $A=\{x_j\}$ and $\{B_2,\ldots,B_m\}=\{\vp_t(F):t\notin\Lambda\}$) gives us the connectedness of
    \[
        F\setminus\{x_j\} = \Big( \bigcup_{t\in\Lambda} \vp_t(F)\setminus\{x_j\} \Big) \cup \Big( \bigcup_{t\notin\Lambda} \vp_t(F) \Big),
    \]
    which contradicts that $x_j\in C_F$. This proves the claim.

    By Corollary~\ref{cor:allverorhor} and Proposition~\ref{prop:twovertical}, $F\setminus\{x_j\}$ contains exactly two vertical connected components. Denote by $\C_{x_j,1}$ (resp. $\C_{x_j,2}$) the component of $F\setminus\{x_j\}$ containing $\vp_j(\C_{x,1})$ (resp. $\vp_j(\C_{x,2})$). Note that
    \begin{align*}
        \C_{x_j,1}\cup\C_{x_j,2} &= F\setminus\{x_j\} = \Big( \bigcup_{i\in\D\setminus\{j\}} \vp_i(F) \Big) \cup \vp_j(F\setminus\{x\}) \\
        &= \Big( \bigcup_{i\in\D\setminus\{j\}} \vp_i(F) \Big) \cup \vp_j(\C_{x,1}) \cup \vp_{j}(\C_{x,2}).
    \end{align*}
    So we can decompose $\D\setminus\{j\}=V_1\cup V_2$ such that
    \[
        \C_{x_j,1} = \Big( \bigcup_{i\in V_1} \vp_i(F) \Big) \cup \vp_j(\C_{x,1}) \quad\text{and}\quad \C_{x_j,2} = \Big( \bigcup_{i\in V_2} \vp_i(F) \Big) \cup \vp_j(\C_{x,2}).
    \]
    For convenience, write $F_{V_p}=\bigcup_{i\in V_p} \vp_i(F)$ for $p=1,2$.

    Again, by Corollary~\ref{cor:allverorhor} and Proposition~\ref{prop:twovertical}, $F\setminus\{y\}$ contains exactly two vertical connected components. Denote the left one by $\C_{y,1}$ and the right one by $\C_{y,2}$. Since $j\pm(0,1)\notin\D$ and $\C_{x,1}$ lies to the left of $\C_{x,2}$,
    \begin{align*}
        \varnothing &= F_{V_1} \cap \vp_j(\C_{x,2}) \\
        &= \Big( \bigcup_{t\in\P:j(t)\in V_1}\vp_{j(t)}(F) \Big)\cap \vp_j(\C_{x,2}) \\
        &= \Big( \bigcup_{t\in\{\nearrow,\rightarrow,\searrow\}:j(t)\in V_1}\vp_{j(t)}(F) \Big)\cap \vp_j(\C_{x,2}) \\
        &= \Big( \bigcup_{t\in\{\nearrow,\rightarrow,\searrow\}:j(t)\in V_1}\vp_{j(t)}(F) \Big)\cap \vp_j(\C_{x,2}\cap(\{1\}\times[0,1])) \\
        &= \Big( \bigcup_{t\in\{\nearrow,\rightarrow,\searrow\}:j(t)\in V_1}\vp_{j(t)}(F) \Big)\cap \vp_j(F\cap(\{1\}\times[0,1])),
    \end{align*}
    where the last equality follows from Lemma~\ref{lem:leftrightside}. Thus
    \[
      \Big( \bigcup_{t\in\{\nearrow,\rightarrow,\searrow\}:j(t)\in V_1}\vp_{j(t)}(F) \Big)\cap \vp_j(F\cap(\{1\}\times[0,1]))=\varnothing.
    \]
    Since Lemma~\ref{lem:leftrightside} also implies that $F\cap(\{1\}\times[0,1])=\C_{y,2}\cap(\{1\}\times[0,1])$, writing $y_j=\vp_j(y)$, we have
    \begin{align*}
        \big( F_{V_1}\setminus\{y_j\} \big) \cap \vp_j(\C_{y,2}) &\subset \Big( \bigcup_{t\in\{\nearrow,\rightarrow,\searrow\}:j(t)\in V_1}\vp_{j(t)}(F) \Big)\cap \vp_j(\C_{y,2}) \\
        &= \Big( \bigcup_{t\in\{\nearrow,\rightarrow,\searrow\}:j(t)\in V_1}\vp_{j(t)}(F) \Big) \cap \vp_j(\C_{y,2}\cap(\{1\}\times[0,1])) \\
        &= \Big( \bigcup_{t\in\{\nearrow,\rightarrow,\searrow\}:j(t)\in V_1}\vp_{j(t)}(F) \Big)\cap \vp_j(F\cap(\{1\}\times[0,1])) = \varnothing.
    \end{align*}
    Similarly, we can show that $\big( F_{V_2}\setminus\{y_j\} \big) \cap \vp_j(\C_{y,1}) = \varnothing$. So the set
    \[
        F\setminus\{y_j\} = \big( \vp_j(\C_{y,1}) \cup \big( F_{V_1}\setminus\{y_j\} \big) \big) \cup \big( \vp_j(\C_{y,2}) \cup \big( F_{V_2}\setminus\{y_j\} \big) \big),
    \]
    as a union of two disjoint closed subsets of $F\setminus\{y_j\}$, should be disconnected. In particular, $y_j\in C_F$. 
\end{proof}

We are now ready to prove Theorem~\ref{thm:2orinfinity}.

\begin{proof}[Proof of Theorem~\ref{thm:2orinfinity}]
    Recall that in the proof of Theorem~\ref{thm:goodcp}, we actually show that there are $x\in C_F$ and $\omega\in\D^*$ such that $\vp_\omega(x)=x$ (please see Remark~\ref{rem:fixedptcutpt}).
    Since we can regard $F$ as the attractor associated with $\{\vp_{\i}: \i\in\D^n\}$ for all $n\geq 1$, we may assume without loss of generality that $\omega\in\D^1$. By Lemma~\ref{lem:iiawayfromj}, $\vp_\omega^2(F)\cap\vp_i(F)=\varnothing$ for all $i\in\D\setminus\{\omega\}$. So $x=\vp_\omega^2(x)\notin\vp_i(F)$ for every such $i$.

    Since $\#C_F\geq 2$, there is some $y\in C_F$ with $y\neq x$. By Proposition~\ref{prop:vpixisthenvpiyis}, $\vp_\omega(y)\in C_F$. Then it follows from an easy induction process that $\vp_\omega^n(y)\in C_F$ for all $n\geq 1$. Note that for any $k\in \mathbb{Z}^+$, $\vp_{\omega}^k$ is an bijection on $\R^2$. Thus, from the facts that $y\neq x$ and $x$ is the fixed point of $\vp_\omega$, $\vp_\omega^n(y)\neq\vp_\omega^m(y)$ whenever $n\neq m$. So we obtain an infinite number of cut points of $F$, which completes the proof.
\end{proof}

\begin{lemma}\label{lem:desiredcoding}
    Let $F$ be a non-fragile connected GSC and let $x\in C_F$. Then there is some infinite word $\bi=i_1i_2\cdots\in\D^\infty$ such that $\{x\}=\bigcap_{n=1}^\infty\vp_{\bi|_n}(F)$ and $\vp_{\bi|_n}^{-1}(x)\in C_F$ for all $n\geq 1$.
\end{lemma}
\begin{proof}
    Let $\Lambda=\{i\in\D: x\in\vp_i(F)\}$. We first show that there is some $i_1\in\Lambda$ such that $x$ is a cut point of $\vp_{i_1}(F)$. Otherwise, $\vp_i(F)\setminus\{x\}$ is connected for every $i\in\D$. Combining with the fact that $x\in C_F$, we can decompose $\D$ as $\D=\D_1\cup\D_2$, where $\D_1\cap\D_2=\varnothing$, such that
    \[
        F\setminus\{x\} = \Big( \bigcup_{i\in\D_1} \vp_i(F)\setminus\{x\} \Big) \cup \Big( \bigcup_{i\in\D_2} \vp_i(F)\setminus\{x\} \Big) =: F_1 \cup F_2
    \]
    is a disjoint union. If $\Lambda\cap\D_1=\varnothing$, then $x\notin F_1$ and hence $F_1=\bigcup_{i\in\D_1}\vp_i(F)$, so the set
    \[
        F = F_1 \cup F_2 \cup \{x\} = F_1 \cup (F_2\cup\{x\}),
    \]
    as a disjoint union of two closed sets, is disconnected. This is a contradiction. So $\Lambda\cap\D_1\neq\varnothing$. Similarly, $\Lambda\cap\D_2\neq\varnothing$. However, this implies that
    \[
        \Big( \bigcup_{i\in\D_1} \vp_i(F) \Big) \cap \Big( \bigcup_{i\in\D_2} \vp_i(F) \Big) = (F_1\cup\{x\}) \cap (F_2\cup\{x\}) = \{x\},
    \]
    which contradicts the fact that $F$ is non-fragile. 
    
    In conclusion, there is some $i_1\in\D$ with $x\in\vp_{i_1}(F)$ such that $\vp_{i_1}(F)\setminus\{x\}$ is disconnected. By the self-similarity, $x_1:=\vp_{i_1}^{-1}(x)\in C_F$. Applying the same argument to the cut point $x_1$ (instead of $x$) as above, we can find some $i_2\in\D$ such that $x_1\in\vp_{i_2}(F)$ and $\vp_{i_2}(F)\setminus\{x_1\}$ is disconnected. So $x_2:=\vp_{i_2}^{-1}(x_1)=\vp_{i_1i_2}^{-1}(x)$ is a cut point of $F$. Repeating this process, we obtain sequences $\{i_n\}_{n=1}^\infty\subset\D$ and $\{x_n\}_{n=1}^\infty\subset F$, where $x_n:=\vp_{i_1\cdots i_n}^{-1}(x)$, such that for every $n\geq 1$, $\vp_{i_{n}}(F)\setminus\{x_{n-1}\}$ is disconnected (with the interpretation $x_0=x$). Thus, writing $\bi=i_1i_2\cdots$, we have
    \[
        \vp_{\bi|_n}^{-1}(x)=\vp_{i_n}^{-1}(\vp_{\bi|_{n-1}}^{-1}(x))=\vp_{i_n}^{-1}(x_{n-1})\in C_F.
    \]
    Furthermore, $x=\vp_{\i|_n}(x_n)\in\vp_{\bi|_n}(F)$ for all $n$ and hence $\{x\}=\bigcap_{n=1}^\infty\vp_{\bi|_n}(F)$. This completes the proof.
\end{proof}

\begin{corollary}\label{cor:unicutptisfixpt}
    Let $F$ be a non-fragile connected GSC. If $\#C_F=1$, then the unique cut point of $F$ is the fixed point of $\vp_i$ for some $i\in\D$.
\end{corollary}

\begin{proof}
    Let $x$ be the unique cut point of $F$. By the above lemma, there is an infinite word $i_1i_2\cdots$ such that $\{x\}=\bigcap_{n=1}^\infty\vp_{i_1\cdots i_n}(F)$ and $\vp_{i_1\cdots i_n}^{-1}(x)\in C_F$ for all $n\geq 1$. So we must have $\vp_{i_1\cdots i_n}^{-1}(x)=x$ for all $n\geq 1$. This is possible only when $i_n\equiv i$ for some $i\in\D$ and $x$ is the fixed point of $\vp_{i}$.
\end{proof}

From Theorem~~\ref{thm:2orinfinity} and Corollary~\ref{cor:unicutptisfixpt}, if there is no digit $i\in\D$ such that the fixed point of $\vp_i$ belongs to $C_F$, then the given non-fragile connected GSC $F$ has either none or infinitely many cut points.

\bigskip

\noindent{\bf Acknowledgements.}
The work of Ruan is supported in part by NSFC grant 11771391, ZJNSF grant LY22A010023 and the Fundamental Research Funds
for the Central Universities of China grant 2021FZZX001-01. The work of Wang is supported in part by the Hong Kong Research Grant Council grants 16308518 and 16317416 and
HK Innovation Technology Fund ITS/044/18FX, as well as Guangdong-Hong Kong-Macao Joint Laboratory for Data-Driven Fluid Mechanics and Engineering Applications. We thank Professors Xin-Rong Dai and Jun Luo for helpful discussions.

\small
\bibliographystyle{amsplain}

\end{document}